\documentclass[a4paper,reqno]{amsart}
\usepackage[utf8]{inputenc}
\usepackage{amssymb}
\usepackage{enumitem}
\usepackage{color}
\usepackage{mathrsfs}
\usepackage{mathtools}
\setlist[enumerate]{label=\emph{(\roman*)}}
\usepackage{hyperref}
\usepackage{verbatim}
\usepackage{environ}
\NewEnviron{rem}{\begin{quote}\small\color{blue}\BODY\end{quote}}
\newtheorem{theorem}{Theorem}[section]

\newtheorem{lemma}[theorem]{Lemma}
\newtheorem{proposition}[theorem]{Proposition}

\theoremstyle{definition}
\newtheorem{definition}[theorem]{Definition}

\numberwithin{equation}{section}

\usepackage{amssymb}
\usepackage{enumitem}
\usepackage{color}
\usepackage{mathrsfs}
\usepackage{mathtools}
\newcommand{\ssubset}{\subset\joinrel\subset}

\newcommand{\N}{\mathbb{N}}

\newcommand{\R}{\mathbb{R}}

\DeclarePairedDelimiter\norm{\lVert}{\rVert}

\renewcommand{\epsilon}{\ensuremath\varepsilon}

\renewcommand{\phi}{\ensuremath{\varphi}}

\newcommand\inner[2]{\langle #1,#2\rangle}

\newcommand\norml[1]{\lVert#1\rVert_{L^2}}
\newcommand\normh[1]{\lVert#1\rVert_{H^{-1}}}
\newcommand\normss[1]{\lVert#1\rVert_{*}}
\newcommand\normds[1]{\lVert#1\rVert_{**}}

\newcommand\jp[1]{\langle#1\rangle}

\newcommand\innerhs[2]{\langle #1,#2\rangle_{\dot{H}^s}}
\newcommand\innerl[2]{\langle #1,#2\rangle_{L^2}}

\newcommand\normc[1]{\lVert#1\rVert_{L^{{2}^*}}}
\newcommand\normcd[1]{\lVert#1\rVert_{L^{(2^*)'}}}
\newcommand\normhs[1]{\lVert#1\rVert_{\dot{H}^s}}
\newcommand\normd[1]{\lVert#1\rVert_{{H}^{-s}}}
\newcommand\pa[1]{\frac{\partial #1}{\partial a}}
\newcommand\pb[1]{\frac{\partial #1}{\partial b}}

\numberwithin{equation}{section}
\begin{document}
	\parindent=0pt
	\title[Stability of HLS inequality under bubbling]{Stability of Hardy Littlewood Sobolev inequality under bubbling}
	\author{Shrey Aryan}
\begin{abstract}
In this note we will generalize the results deduced in \cite{FigGla20} and \cite{deng2021sharp} to fractional Sobolev spaces. In particular we will show that for $s\in (0,1)$, $n>2s$ and $\nu\in \mathbb{N}$ there exists constants $\delta = \delta(n,s,\nu)>0$ and $C=C(n,s,\nu)>0$ such that for any function $u\in \dot{H}^s(\R^n)$ satisfying, 
\begin{align*}
\left\| u-\sum_{i=1}^{\nu} \tilde{U}_{i}\right\|_{\dot{H}^s} \leq \delta
\end{align*}
where $\tilde{U}_{1}, \tilde{U}_{2},\cdots \tilde{U}_{\nu}$ is a $\delta-$interacting family of Talenti bubbles, there exists a family of Talenti bubbles $U_{1}, U_{2},\cdots U_{\nu}$ such that
\begin{align*}
\left\| u-\sum_{i=1}^{\nu} U_{i}\right\|_{\dot{H}^s} \leq C\left\{\begin{array}{ll}
\Gamma & \text { if } 2s < n < 6s,\\
\Gamma|\log \Gamma|^{\frac{1}{2}} & \text { if } n=6s, \\
\Gamma^{\frac{p}{2}} & \text { if } n > 6s
\end{array}\right.
\end{align*}
for $\Gamma=\left\|\Delta u+u|u|^{p-1}\right\|_{H^{-s}}$ and $p=2^*-1=\frac{n+2s}{n-2s}.$
\end{abstract}
\maketitle
\section{Introduction}
\subsection{Background} Recall that the Sobolev inequality for exponent $2$ states that there exists a positive constant $S>0$ such that 
\begin{align}\label{sobolev-ineq}
    S \normc{u} \leq \norml{\nabla u}
\end{align}
for $n\geq 3$, $u\in C^{\infty}_c(\R^n)$ and $2^{*}=\frac{2n}{n-2}.$ By density, one can extend this inequality to functions in the homogeneous Sobolev space $\dot{H}^1(\R^n)$. Talenti \cite{Talenti} and Aubin \cite{Aubin} independently computed the optimal $S>0$ and showed that the extremizers of \eqref{sobolev-ineq} are functions of the form
\begin{align}\label{talenti-bubble-eqn}
    U[z,\lambda](x) = c_{n}\left(\frac{\lambda}{1+\lambda^2|x-z|^2}\right)^{(n-2)/2},
\end{align}
where $c_n>0$ is a dimenion dependent constant. Furthermore, it is well known that the Euler-Lagrange equation associated with the inequality \eqref{sobolev-ineq} is
\begin{align}\label{euler-pde}
 \Delta u + u|u|^{2^*-2}=0.
\end{align}
In particular, Caffarelli, Gidas, and Spruck in \cite{caffarelli-gidas-spruck} showed that the Talenti bubbles are the only positive solutions of the equation
\begin{align}\label{bubble-pde}
 \Delta u + u^{p}=0,
\end{align}
where $p=2^{*}-1.$ Geometrically, equation \eqref{bubble-pde} arises in the context of the Yamabe problem which given a Riemannian manifold $(M,g_0)$, in dimension $n\geq 3$, asks for a metric $g$ conformal to $g_0$ with prescribed scalar curvature equal to some function $R.$ In the case when $(M,g_0)$ is the round sphere and $R$ is a positive constant then the result of Caffarelli, Gidas, and Spruck can be seen as a rigidity result in the sense that the only possible choice of conformal functions are the Talenti bubbles defined in \eqref{talenti-bubble-eqn}. For more details regarding the geometric perspective see the discussion in Section 1.1 in \cite{ciralo}.

Following the rigidity result discussed above, it is natural to consider the case of almost rigidity. Here one asks if being an \textit{approximate} solution of \eqref{euler-pde} implies that $u$ is \textit{close} to a Talenti bubble. We first observe that this statement is not always true. For instance, consider two bubbles $U_1 = U[Re_1,1]$ and $U_2=U[-Re_1,1]$, where $R\gg 1$ and $e_1=(1,0,\cdots,0)\in \R^n$. Then $u=U_1+U_2$ will be an \textit{approximate} solution of \eqref{euler-pde} in some well-defined sense, however, it is clear that $u$ is not close to either of the two bubbles. Thus, it is possible that if $u$ \textit{approximately} solves, \eqref{euler-pde} then it might be \textit{close} to a sum of Talenti bubbles. However, Struwe in \cite{Struwe1984} showed that the above case, also known as \textit{bubbling}, is the only possible worst case. More, precisely, he showed the following
\begin{theorem}[Struwe]\label{Struwe-theorem}
Let $n \geq 3$ and $\nu \geq 1$ be positive integers. Let $\left(u_{k}\right)_{k \in \mathbb{N}} \subseteq \dot{H}^{1}\left(\mathbb{R}^{n}\right)$ be a sequence of nonnegative functions such that 
\begin{align}\label{energy-constraint}
\left(\nu-\frac{1}{2}\right) S^{n} \leq \int_{\mathbb{R}^{n}}\left|\nabla u_{k}\right|^{2} \leq\left(\nu+\frac{1}{2}\right) S^{n}
\end{align}
with $S=S(n)$ being the sharp constant for the Sobolev inequality defined in \eqref{sobolev-ineq}. Assume that
\begin{align}\label{delta-convg}
 \left\|\Delta u_{k}+u_{k}^{2^{*}-1}\right\|_{H^{-1}} \rightarrow 0 \quad \text { as } k \rightarrow \infty \text { . }   
\end{align}
Then there exist a sequence $\left(z_{1}^{(k)}, \ldots, z_{\nu}^{(k)}\right)_{k \in \mathbb{N}}$ of $\nu$ -tuples of points in $\mathbb{R}^{n}$ and a sequence $\left(\lambda_{1}^{(k)}, \ldots, \lambda_{\nu}^{(k)}\right)_{k \in \mathbb{N}}$
of $\nu$ -tuples of positive real numbers such that
\begin{align}\label{seq-convg}
 \left\|\nabla\left(u_{k}-\sum_{i=1}^{\nu} U\left[z_{i}^{(k)}, \lambda_{i}^{(k)}\right]\right)\right\|_{L^{2}} \rightarrow 0 \quad \text { as } k \rightarrow \infty.   
\end{align}
\end{theorem}
Morally, the above statement says that if the energy of a function $u_k$ is less than or equal to the energy of $\nu$ bubbles as in \eqref{energy-constraint} and if $u_k$ almost solves \eqref{euler-pde} in the sense that the deficit \eqref{seq-convg} is small then $u_k$ is close to $\nu$ bubbles in $\dot{H}^1$ sense. For instance, when $\nu=1$ the energy constraint in \eqref{energy-constraint} forbids us from writing $u$ as the sum of two bubbles since that would imply that $u$ has energy equal to $2S^n$. Thus, Struwe's result gives us a qualitative answer to the almost rigidity problem posed earlier. 

Building upon this work, Ciraolo, Figalli, and Maggi in \cite{ciralo} proved the first sharp quantitative stability result around one bubble, i.e. the case when $\nu=1$. Their result implied that the distance between the bubbles \eqref{seq-convg} should be linearly controlled by the deficit as in \eqref{delta-convg}. Thus one might be inclined to conjecture that the same should hold in the case of more than one bubble or when $\nu\geq 2.$

Figalli and Glaudo in \cite{FigGla20} investigated this problem and gave a positive result for any dimension $n\in \N$ when the number of bubbles $\nu = 1$ and for dimension $3\leq n \leq 5$ when the number of bubbles $\nu\geq 2.$ More precisely they proved the following
\begin{theorem}[Figalli-Glaudo]\label{theorem-delta-multi-bubble}
Let the dimension $3\leq n\leq 5$ when $\nu \geq 2$ or $n\in \N$ when $\nu=1$. Then there exists a small constant $\delta(n,\nu) > 0$ and constant $C(n,\nu)>0$ such that the following statement holds. Let $u\in \dot{H}^1(\R^n)$ such that
\begin{align}\label{multi-bubble-delta}
    \norm{\nabla u - \sum_{i=1}^{\nu}\nabla \Tilde{U}_i}_{L^2} \leq \delta
\end{align}
where $\left(\tilde{U}_{i}\right)_{1 \leq i \leq \nu}$ is any family of $\delta-$interacting Talenti bubbles. Then there exists a family of Talenti bubbles $(U_i)_{i=1}^{\nu}$ such that
\begin{align}\label{multi-bubble-main-estimate}
        \norm{\nabla u - \sum_{i=1}^{\nu}\nabla U_i}_{L^2} \leq C \normh{\Delta u + u|u|^{p-1}},
\end{align}
where $p=2^{*}-1=\frac{n+2}{n-2}.$
\end{theorem}
Here a $\delta-$ interacting family is a family of bubbles $(U_i[z_i, \lambda_i])_{i=1}^{\nu}$ where $\nu \geq 1$ and $\delta > 0$ if
\begin{align}\label{delta-interacting-bubble}
    Q = \min_{i\neq j}\left(\frac{\lambda_i}{\lambda_j}, \frac{\lambda_j}{\lambda_i}, \frac{1}{\lambda_i\lambda_j|z_i-z_j|^2} \right) \leq \delta.
\end{align}
Note that the definition $\delta$-interaction between bubbles follows naturally by estimating the $\dot{H}^1$ interaction between any two bubbles. This is explained in Remark 3.2 in \cite{FigGla20} which we recall here for the reader's convenience. If $U_1=U\left[z_1, \lambda_1\right]$ and $U_2=U\left[z_2, \lambda_2\right]$ are two bubbles,  then from the interaction estimate in Proposition B.2 in \cite{FigGla20} and the fact that$-\Delta U_i=U_i^{2^*-1}$ for $i=1,2$ we have
$$
\int_{\mathbb{R}^n} \nabla U_1 \cdot \nabla U_2=\int_{\mathbb{R}^n} U_1^{2^*-1} U_2 \approx \min \left(\frac{\lambda_1}{\lambda_2}, \frac{\lambda_2}{\lambda_1}, \frac{1}{\lambda_1 \lambda_2\left|z_1-z_2\right|^2}\right)^{\frac{n-2}{2}} .
$$
In particular, if $U_1$ and $U_2$ belong to a $\delta$-interacting family then their $\dot{H}^1$-scalar product is bounded by $\delta^{\frac{n-2}{2}}$. 

In the proof of the above theorem, the authors first approximate the function $u$ by a linear combination $\sigma = \sum_{i=1}^\nu \alpha_i U_i[z_i,\lambda_i]$ where $\alpha_i$ are real-valued coefficients. They then show that each coefficient $\alpha_i$ is close to $1$. To quantify this notion we say that the family $(\alpha_i, U_i[z_i, \lambda_i])_{i=1}^{\nu}$ is $\delta-$interacting if \eqref{delta-interacting-bubble} holds and 
\begin{align}
    |\alpha_i-1|\leq \delta. 
\end{align}
The authors \cite{FigGla20} also constructed counter-examples for the higher dimensional case when $n\geq 6$ proving that the estimate \eqref{multi-bubble-main-estimate} does not hold. The optimal estimates in dimension $n\geq 6$ were recently established by Deng, Sun, and Wei in \cite{deng2021sharp}. They proved the following
\begin{theorem}[Deng-Sun-Wei]\label{deng-theorem-main}
Let the dimension $n \geq 6 $ and the number of bubbles $\nu\geq 2$. Then there exists $\delta=\delta(n, \nu)>0$ and a large constant $C=C(n,\nu)>0$ such that the following holds. Let $u \in \dot{H}^{1}\left(\mathbb{R}^{n}\right)$ be a function such that
\begin{align}\label{deng-delta-condition}
\norml{\nabla u-\sum_{i=1}^{\nu} \nabla \tilde{U}_{i}} \leq \delta
\end{align}
where $\left(\tilde{U}_{i}\right)_{1 \leq i \leq \nu}$ is a $\delta-$interacting family of Talenti bubbles. Then there exists a family of Talenti bubbles $(U_i)_{i=1}^{\nu}$ such that
\begin{align}\label{deng-main-estimate}
\norml{\nabla u-\sum_{i=1}^{\nu} \nabla U_{i}} \leq C\left\{\begin{array}{ll}
\Gamma|\log \Gamma|^{\frac{1}{2}} & \text { if } n=6, \\
\Gamma^{\frac{p}{2}} & \text { if } n \geq 7
\end{array}\right.
\end{align}
for $\Gamma=\left\|\Delta u+u|u|^{p-1}\right\|_{H^{-1}}$ and $p=2^{*}-1=\frac{n+2}{n-2}.$
\end{theorem}
Given these results, it is natural to wonder if they generalize to some reasonable setting. One possible way is to observe that the Sobolev inequality \eqref{sobolev-ineq} can be generalized to hold for functions in fractional spaces. This is better known as the Hardy-Littlewood-Sobolev (HLS) inequality, which states that there exists a positive constant $S>0$ such that for all $u\in C^{\infty}_{c}(\R^n)$, we have
\begin{align}\label{eqn:hls-ineq}
    S \normc{u} \leq \norml{(-\Delta)^{s/2}u} 
\end{align}
where $s\in (0,1)$ and $n>2s$ and $2^* = \frac{2n}{n-2s}.$\footnote{Note that we abuse notation by denoting $S$, $2^*$ and $p=2^{*}-1$ to be the fractional analogs of the best constant and the critical exponents related to the Sobolev inequality. For the remainder of this paper, we will only use the fractional version of these constants.}  
Observe that by a  density argument, the HLS inequality also holds for all functions $u\in \dot{H}^s(\R^n),$ where we define the fractional homogeneous Sobolev space $\dot{H}^s(\R^n)=\dot{W}^{s,2}(\R^n)$ as the closure of the space of test functions $C_c^{\infty}(\R^n)$ with respect to the norm
\begin{align*}
    \norm{u}_{\dot{H}^s}^2 = \norml{(-\Delta)^{s/2} u}^2 = \frac{C_{n,s}}{2} \int_{\R^n} \int_{\R^n} \frac{|u(x)-u(y)|^2}{|x-y|^{n+2s}} dx dy
\end{align*}
equipped with the natural inner product for $u,v\in \dot{H}^s$
\begin{align*}
\inner{u}{v}_{\dot{H}^s} =  \frac{C_{n,s}}{2}\int_{\R^n} \int_{\R^n} \frac{(u(x)-u(y)) (v(x)-v(y)) }{|x-y|^{n+2s}} dx dy, 
\end{align*}
where $C_{n,s}>0$ is a constant depending on $n$ and $s.$ Furthermore Lieb in \cite{lieb} found the optimal constant and established that the extremizers of \eqref{eqn:hls-ineq} are functions of the form
\begin{align}\label{eqn:fractional-bubble}
    U[z,\lambda](x) = c_{n,s} \left(\frac{\lambda}{1+ \lambda^2|x-z|^2}\right)^{(n-2s)/2}
\end{align}
for some $\lambda>0$ and $z\in \R^n$, where we choose the constant $c_{n,s}$ such that the bubble $U(x)=U[z,\lambda](x)$ satisfies
\begin{align*}
    \int_{\R^n} |(-\Delta)^{s/2}U|^2 = \int_{\R^n} U^{2^*}=S^{n/s}.
\end{align*}
Similar to the Sobolev inequality, the Euler Lagrange equation associated with \eqref{eqn:hls-ineq} is as follows
\begin{align}\label{eqn:fractional-euler}
    (-\Delta)^s u = u^{p},
\end{align}
where $p=2^*-1=\frac{n+2s}{n-2s}.$ Chen, Li, and Ou in \cite{fractional-euler} showed that the only positive solutions to \ref{eqn:fractional-euler} are the bubbles described in \eqref{eqn:fractional-bubble}. Following this rigidity result Palatucci and Pisante in \cite{fractional-struwe} proved an analog of Struwe's result on any bounded subset of $\R^n$. However, as noted in \cite{denitti2022stability} the same proof works carries over to $\mathbb{R}^n$. We state the result in the multi-bubble case (i.e. when $\nu\geq 1$). See also Lemma 2.1 in \cite{denitti2022stability} for the single bubble case.
\begin{theorem}[Palatucci-Pisante, Nitti-König]\label{thm:qual-fractional-struwe}
Let $n \in \mathbb{N}, 0<s<n / 2$, $\nu \geq 1$ be positive integers, and $\left(u_{k}\right)_{k \in \mathbb{N}} \subseteq \dot{H}^{s}\left(\mathbb{R}^{n}\right)$ be a sequence of functions such that 
\begin{align}\label{fractional-energy-constraint}
\left(\nu-\frac{1}{2}\right) S^{n/s} \leq \int_{\mathbb{R}^{n}}\left|(-\Delta)^{s/2} u_{k}\right|^{2} \leq\left(\nu+\frac{1}{2}\right)  S^{n/s}
\end{align}
with $S=S(n,s)$ being the sharp constant for the HLS inequality defined in \eqref{eqn:hls-ineq}. Assume that
\begin{align}\label{fractional-delta-convg}
 \left\|\Delta u_{k}+u_{k}^{2^{*}-1}\right\|_{\dot{H}^{-s}} \rightarrow 0 \quad \text { as } k \rightarrow \infty \text { . }   
\end{align}
Then there exist a sequence $\left(z_{1}^{(k)}, \ldots, z_{\nu}^{(k)}\right)_{k \in \mathbb{N}}$ of $\nu$ -tuples of points in $\mathbb{R}^{n}$ and a sequence $\left(\lambda_{1}^{(k)}, \ldots, \lambda_{\nu}^{(k)}\right)_{k \in \mathbb{N}}$
of $\nu$ -tuples of positive real numbers such that
\begin{align}\label{fractional-seq-convg}
 \left\|u_{k}-\sum_{i=1}^{\nu} U\left[z_{i}^{(k)}, \lambda_{i}^{(k)}\right]\right)\|_{\dot{H}^{s}} \rightarrow 0 \quad \text { as } k \rightarrow \infty.   
\end{align}
\end{theorem}
Thus following the above qualitative result, we establish the following sharp quantitative almost rigidity for the critical points of the HLS inequality
\begin{theorem}\label{fractional-theorem-main}
Let the dimension $n >2s$ where $s\in (0,1)$ and the number of bubbles $\nu \geq 1$. Then there exists $\delta=\delta(n,\nu, s)>0$ and a large constant $C=C(n,\nu, s)>0$ such that the following holds. Let $u \in \dot{H}^{s}\left(\mathbb{R}^{n}\right)$ be a function such that
\begin{align}\label{fractional-delta-condition}
\left\| u-\sum_{i=1}^{\nu} \tilde{U}_{i}\right\|_{\dot{H}^s} \leq \delta
\end{align}
where $\left(\tilde{U}_{i}\right)_{1 \leq i \leq \nu}$ is a $\delta-$interacting family of Talenti bubbles. Then there exists a family of Talenti bubbles $\left({U}_{i}\right)_{1 \leq i \leq \nu}$ such that
\begin{align}\label{fractional-main-estimate}
\left\| u-\sum_{i=1}^{\nu} U_{i}\right\|_{\dot{H}^s} \leq C\left\{\begin{array}{ll}
\Gamma & \text { if } 2s < n < 6s,\\
\Gamma |\log \Gamma|^{\frac{1}{2}} & \text { if } n =    6s\\
\Gamma^{\frac{p}{2}} & \text { if } n >  6s
\end{array}\right.
\end{align}
for $\Gamma=\left\|\Delta u+u|u|^{p-1}\right\|_{H^{-s}}$ and $p=2^*-1=\frac{n+2s}{n-2s}.$
\end{theorem}
Note that analogous to the local case, i.e. when $s=1$, the definition of $\delta$-interaction between bubbles carries over naturally since given any two bubbles $U_1=U\left[z_1, \lambda_1\right], U_2=U\left[z_2, \lambda_2\right]$ we can estimate their interaction using Proposition B.2 in \cite{FigGla20} and the fact that$(-\Delta)^s U_i=U_i^p$ for $i=1,2$ we have
$$
\int_{\mathbb{R}^n} (-\Delta)^{s/2} U_1 \cdot (-\Delta)^{s/2}  U_2=\int_{\mathbb{R}^n} U_1^p U_2 \approx \min \left(\frac{\lambda_1}{\lambda_2}, \frac{\lambda_2}{\lambda_1}, \frac{1}{\lambda_1 \lambda_2\left|z_1-z_2\right|^2}\right)^{\frac{n-2s}{2}} .
$$
In particular, if $U_1$ and $U_2$ belong to a $\delta$-interacting family then their $\dot{H}^s$-scalar product is bounded by $\delta^{\frac{n-2s}{2}}$. Furthermore, the conditions on the scaling and the translation parameters in Theorem 1.1 of Palatucci and Pisante in \cite{fractional-struwe} imply that bubbling observed in the non-local setting is caused by the same mechanism as in the local case, i.e. either when the distance between the centers of the bubbles or their relative concentration scales (such as $\lambda_i/\lambda_j$ in \ref{defn: Q interaction term}) tend to infinity.

The above theorem, besides being a natural generalization, also provides a quantitative rate of convergence to equilibrium for the non-local fast diffusion equation
\begin{align}\label{eqn:frac-fast-diff}
\begin{cases}\partial_t u+(-\Delta)^s\left(|u|^{m-1} u\right)=0, & t>0, x \in \mathbb{R}^n \\ u(0, x)=u_0(x), & x \in \mathbb{R}^n\end{cases} 
\end{align}
with
$$
m:=\frac{1}{p}=\frac{1}{2^*-1}=\frac{n-2 s}{n+2 s}.
$$
In particular, Nitti and König in \cite{denitti2022stability} established the following result
\begin{theorem}[Nitti-König]\label{diff-thm}
Let $n \in \mathbb{N}$ and $s \in(0, \min \{1, n / 2\})$. Let $u_0 \in C^2\left(\mathbb{R}^n\right)$ such that $u_0 \geq 0$ and $|x|^{2 s-n} u_0^m\left(x /|x|^2\right)$ can be extended to a positive $C^2$ function in the origin. Then, there exist an extinction time $\bar{T}=T\left(u_0\right) \in(0, \infty)$ such that the solution $u$ of \eqref{eqn:frac-fast-diff} satisfies $u(t, x)>0$ for $t \in(0, \bar{T})$ and $u(\bar{T}, \cdot) \equiv 0$. Moreover, there exist $z \in \mathbb{R}^n$ and $\lambda>0$, such that
$$
\left\|\frac{u(t, \cdot)}{U_{\bar{T}, z, \lambda}(t, \cdot)}-1\right\|_{L^{\infty}\left(\mathbb{R}^n\right)} \leq C_*(\bar{T}-t)^\kappa \quad \forall t \in(0, \bar{T}),
$$
for some $C_*>0$ depending on the initial datum $u_0$ and all $\kappa<\kappa_{n, s}$, where
$$
\kappa_{n, s}:=\frac{1}{(n+2-2 s)(p-1)} \gamma_{n, s}^2
$$
with $p=\frac{n+2 s}{n-2 s}$, $\gamma_{n, s}=\frac{4 s}{n+2 s+2}$ and 
$$
U_{\bar{T}, z, \lambda}(t, x):=\left(\frac{p-1}{p}\right)^{\frac{p}{p-1}}(\bar{T}-t)^{\frac{p}{p-1}} U[z, \lambda]^p(x), \quad t \in(0, \bar{T}), x \in \mathbb{R}^n
$$
with $\frac{p}{p-1}=\frac{1}{1-m}=\frac{n+2 s}{4 s}$.
\end{theorem}
Using Theorem \ref{fractional-theorem-main} one can also establish Theorem \ref{diff-thm} by following the arguments of the proof of Theorem 5.1 in \cite{FigGla20}. However, such a result would not provide explicit bounds on the rate parameter $\kappa$. The remarkable point of the Theorem \ref{diff-thm} is that the authors provide an explicit bound on the parameter $\kappa.$

\subsection{Proof Sketch and Challenges}\label{section:proof-sketch}
The proof of the Theorem \ref{fractional-theorem-main} follows by adapting the arguments in \cite{FigGla20} and \cite{deng2021sharp} however we need to overcome several difficulties introduced by the non-locality of the fractional laplacian.

To understand the technical challenges involved, we sketch the proof of Theorem \ref{fractional-theorem-main}. First, consider the case when $n\in (2s,6s)$. Then the linear estimate in \eqref{fractional-main-estimate} follows by adapting the arguments in \cite{FigGla20}. The starting point in their argument is to approximate the function $u$ by a linear combination of Talenti Bubbles. This can be achieved for instance by solving the following minimization problem
\begin{align*}
    \normhs{u -\sum_{i=1}^{\nu}\alpha_i  U[z_i, \lambda_i] }= \min_{\bar{z}_{1}, \ldots, \bar{z}_{\nu} \in \mathbb{R}^n \atop \bar{\lambda}_{1}, \ldots, \bar{\lambda}_{\nu}\in \mathbb{R}_{+}, \tilde{\alpha}_1,\ldots,\tilde{\alpha}_{\nu} \in \mathbb{R}}\normhs{ u-\sum_{i=1}^{\nu}\tilde{\alpha_i}  U[\tilde{z}_i, \tilde{\lambda}_i]}
\end{align*}
where $U_i = U[z_i, \lambda_i]$ and $\sigma = \sum_{i=1}^{\nu} \alpha_i U_i$ is the sum of Talenti Bubbles closest to the function $u$ in the $\dot{H}^{s}(\R^n)$ norm. Denote the error between $u$ and this approximation by $\rho$, i.e.
\begin{align*}
    u = \sigma + \rho.
\end{align*}
Now, the idea is to estimate the $H^s$ norm of $\rho$. Thus
\begin{align*}
    \normhs{\rho}^2=\innerhs{\rho}{\rho} &= \innerhs{\rho}{u-\sigma} =\innerhs{\rho}{u} = \innerl{\rho}{(-\Delta)^s u} \\
    &= \innerl{\rho}{(-\Delta)^s u-u|u|^{p-1}} + \innerl{\rho}{u|u|^{p-1}}\\
    &\leq \normhs{\rho} \normd{(-\Delta)^s u-u|u|^{p-1}} + \int_{\R^n} u|u|^{p-1}\rho.
\end{align*}
where in the third equality we made use of the orthogonality between $\rho$ and $\sigma$. This follows from differentiating in the coefficients $\alpha_i$. On expanding $u=\sigma+\rho$ we can further estimate the second term in the final inequality as follows
\begin{align*}
\int_{\R^n} \rho u|u|^{p-1} &\leq  p\int_{\R^n} |\sigma|^{p-1} \rho^2 \nonumber \\
&\quad + C_{n,\nu}\left(\int_{\R^n}|\sigma|^{p-2} |\rho|^3 + \int_{\R^n}|\rho|^{p+1}+\sum_{1\leq i\neq j\leq \nu} \int_{\R^n}\rho U_i^{p-1} U_j\right).
\end{align*}
The argument for controlling the second and third terms follows in the same way as in \cite{FigGla20} and this is where one uses the fact that $n<6s.$ However, it is not so clear how to estimate the first and the last term. For the first term, ideally, we would like to show that
\begin{align*}
    p \int_{\R^n} |\sigma|^{p-1} \rho^2 \leq \tilde{c} \normhs{\rho}^2
\end{align*}
where $\Tilde{c}<1$ is a positive constant. In \cite{FigGla20}, the authors make use of a spectral argument that essentially says that the third eigenvalue of a linearized operator associated with \eqref{eqn:fractional-euler} is strictly larger than $p.$ Such a result was not proved in the non-local setting and thus, our first contribution is to prove this fact rigorously in Section \ref{section:spectral}.

By linearizing around a single bubble we deduce that 
\begin{align*}
    p \int_{\R^n} |U|^{p-1} \rho^2 \leq \tilde{c} \normhs{\rho}^2
\end{align*}
however, recall that our goal is to estimate $\int_{\R^n} |\sigma|^{p-1} \rho^2$ instead of $\int_{\R^n} |U|^{p-1} \rho^2.$ Thus one would like to localize $\sigma$ by a bump function $\Phi_i$ such that $\sigma \Phi_i \approx \alpha_i  U_i.$ This allows us to show that
\begin{align*}
    \int_{\mathbb{R}^{n}}\left(\rho \Phi_{i}\right)^{2} U_{i}^{p-1} \leq \frac{1}{\Lambda} \int_{\mathbb{R}^{n}}\left|(-\Delta)^{s/2}\left(\rho \Phi_{i}\right)\right|^{2}+o(1)\normhs{\rho}^{2}
\end{align*}
where $\Lambda>p.$ Now observe that in the local setting, estimating the first term on the right-hand side of the above inequality would be elementary since we could simply write the following identity
\begin{align*}
\int_{\mathbb{R}^n}\left|\nabla\left(\rho \Phi_i\right)\right|^2=\int_{\mathbb{R}^n}|\nabla \rho|^2 \Phi_i^2+\int_{\mathbb{R}^n} \rho^2\left|\nabla \Phi_i\right|^2+2 \int_{\mathbb{R}^n} \rho \Phi_i \nabla \rho \cdot \nabla \Phi_i.
\end{align*}
However, this clearly fails in the non-local setting. To get around this issue we observe that
\begin{align*}
   (-\Delta)^{s/2}(\rho \Phi_i) = \rho(-\Delta)^{s/2}\Phi_i +\Phi_i(-\Delta)^{s/2}\rho+\mathcal{C}(\rho, \Phi_i)
\end{align*}
where $\mathcal{C}(\rho,\Phi_i)$ is error term introduced by the fractional laplacian. Thus to estimate the product of the fractional gradient we must control this error term, which is fortunately controlled by the commutator estimates from Theorem A.8 in \cite{KeningProductRule}. These estimates state that the remainder term $\mathcal{C}(\rho, \Phi_i)$ satisfies
\begin{align}\label{kato}
\norml{\mathcal{C}(\rho, \Phi_i}\leq C\norm{(-\Delta)^{s_{1} / 2} \Phi_i}_{L^{p_{1}}}
\norm{(-\Delta)^{s_{2} / 2} \rho}_{L^{p_{2}}}
\end{align}
provided that $s_{1}, s_{2} \in [0, s], s=s_{1}+s_{2}$ and $p_{1}, p_{2} \in(1,+\infty)$ satisfy
$$
\frac{1}{2}=\frac{1}{p_{1}}+\frac{1}{p_{2}}.
$$
Setting $s_1=s, s_2=0, p_1=n/s$ and $p_2=2^* = \frac{2n}{n-2s}$ we get
\begin{align}\label{fractional-final-error-estimate}
\norml{\mathcal{C}(\rho, \Phi_i}\leq C\norm{(-\Delta)^{s / 2} \Phi_i}_{L^{n/s}}\norm{\rho}_{L^{2^*}}\lesssim \norm{(-\Delta)^{s / 2} \Phi_i}_{L^{n/s}}\normhs{\rho}.
\end{align}
Thus it remains to control the fractional gradient of the bump function which we establish in Lemma \ref{fractional-cutoff-lemma}.

Next, we consider the case when $n\geq 6s$. We follow the proof strategy of \cite{deng2021sharp}. The argument proceeds in the following steps. 

Step (\rm{i}). First we obtain a family of bubbles $(U_i)_{i=1}^{\nu}$ as a result of solving the following minimization problem
\begin{align}\label{minimization process}
    \normhs{ u -  \sum_{i=1}^{\nu}U[z_i, \lambda_i]} =\inf _{\tilde{z}_{1},\cdots,\tilde{z}_{\nu} \in \mathbb{R}^{n} \atop \tilde{\lambda}_{1},\cdots, \tilde{\lambda}_{\nu}>0}\normhs{ u -  \sum_{i=1}^{\nu}U[\tilde{z}_i, \tilde{\lambda}_i]}
\end{align}
and denote the error $\rho = u - \sum_{i=1}^{\nu}U_i= u - \sigma.$ Since $(-\Delta)^s U_i = U_i^p$ for $i=1,\cdots,\nu$ we have
\begin{align}\label{u-split}
(-\Delta)^s u-u|u|^{p-1}
&=(-\Delta)^s(\sigma+\rho)-(\sigma+\rho)|\sigma+\rho|^{p-1} =(-\Delta)^s \rho-p\sigma^{p-1}\rho \nonumber\\
&\quad -(\sigma^p - \sum_{i=1}^{\nu}U_i^p) - ((\sigma+\rho)|\sigma+\rho|^{p-1}-\sigma^p-p\sigma^{p-1}\rho).
\end{align}
Thus the error $\rho$ satisfies
\begin{align}\label{rho PDE}
(-\Delta)^s\rho-p\sigma^{p-1}\rho-I_1-I_2-f=0,
\end{align}
where 
\begin{align}\label{f I1 I2 defn}
    f &= (-\Delta)^s u -u|u|^{p-1},\quad  I_1=\sigma^p- \sum_{i=1}^{\nu}U_i^p, \nonumber \\ I_2&=(\sigma+\rho)|\sigma+\rho|^{p-1}-\sigma^p-p\sigma^{p-1}\rho.
\end{align}
Clearly \eqref{fractional-delta-condition} implies that $\normhs{ \rho}\leq \delta$ and furthermore the family of bubbles $(U_i)_{i=1}^{\nu}$ are also $\delta'-$interacting where $\delta'$ tends to $0$ as $\delta$ tends to $0.$ In the first step we decompose $\rho = \rho_0 + \rho_1$ where we will prove the existence of the first approximation $\rho_0$, which solves the following system
\begin{align}
\left\{\begin{array}{l}(-\Delta)^s \rho_{0}-\left[\left(\sigma+\rho_{0}\right)^{p}-\sigma^{p}\right]=\left(\sigma^{p}-\sum_{j=1}^{\nu} U_{j}^{p}\right)-\sum_{i=1}^{\nu} \sum_{a=1}^{n+1} c_{a}^{i} U_{i}^{p-1} Z_{i}^{a} \\ \innerhs{\rho}{Z_i^a}=0, \quad i=1,\cdots,\nu;\quad a=1, \cdots, n+1\end{array}\right.
\end{align}
where $(c_a^i)$ is a family of scalars and $Z^a_i$ are the rescaled derivative of $U[z_i,\lambda_i]$ defined as follows
\begin{align}\label{deng-z-derivative-notation}
Z_{i}^{a}=\frac{1}{\lambda_{i}} \frac{\partial U\left[z, \lambda_{i}\right]}{\partial z^{a}}\bigg|_{z=z_{i}}, \quad Z_{i}^{n+1}=\lambda_{i} \frac{\partial U\left[z_{i}, \lambda\right]}{\partial \lambda}\bigg|_{\lambda=\lambda_{i}}    
\end{align}
for $1\leq i\leq \nu$ and $1\leq a\leq n.$
\newline
Step (\rm{ii}). The next step is to establish point-wise estimates on $\rho_0.$ In order to do this we argue as in the proof of Proposition 3.3 in \cite{deng2021sharp} to show that 
\begin{align*}
    |I_1|=|\sigma^p-\sum_{i=1}^{\nu}U_i^p|\lesssim V(x) = \sum_{i=1}^{\nu}\left(\frac{\lambda_{i}^{\frac{n+2s}{2}} R^{2s-n}}{\left\langle y_{i}\right\rangle^{4s}} \chi_{\left\{\left|y_{i}\right| \leq R\right\}}+\frac{\lambda_{i}^{\frac{n+2s}{2}} R^{-4s}}{\left\langle y_{i}\right\rangle^{n-2s}} \chi_{\left\{\left|y_{i}\right| \geq R / 2\right\}}\right)
\end{align*}
which along with an a priori estimate and a fixed point argument allows us to show that
\begin{align*}
    |\rho_0(x)|\lesssim W(x) =\sum_{i=1}^{\nu}\left(\frac{\lambda_{i}^{\frac{n-2s}{2}} R^{2s-n}}{\left\langle y_{i}\right\rangle^{2s}} \chi_{\left\{\left|y_{i}\right| \leq R\right\}}+\frac{\lambda_{i}^{\frac{n-2s}{2}} R^{-4s}}{\left\langle y_{i}\right\rangle^{n-4s}} \chi_{\left\{\left|y_{i}\right| \geq R / 2\right\}}\right)
\end{align*}
where $y_i=\lambda_i(x-z_i)$ and
\begin{align*}
R = \min_{i\neq j}\frac{R_{ij}}{2} = \frac{1}{2}\min_{i\neq j}\max_{j=1,\cdots,\nu}\left(\sqrt{\lambda_{i} / \lambda_{j}}, \sqrt{\lambda_{j} / \lambda_{i}}, \sqrt{\lambda_{i} \lambda_{j}}\left|z_{i}-z_{j}\right|\right).    
\end{align*}
This will also give us a control on the energy of $\rho_0$, in particular arguing as in the proof of Proposition 3.9 in \cite{deng2021sharp} one can show that
\begin{align*}
    \norml{(-\Delta)^{s/2}\rho_0}\lesssim \begin{cases}Q|\log Q|^{\frac12}, &\text{if }n=6s,\\ Q^{\frac{p}{2}},&\text{if }n> 6s.\end{cases}
\end{align*}
Here we define the interaction term $Q$ by first defining the interaction between two bubbles $U_i$ and $U_j$
\begin{align}\label{defn: qij interaction term}
    q_{ij} = \left( \frac{\lambda_i}{\lambda_j} + \frac{\lambda_j}{\lambda_i} + \lambda_i\lambda_j|z_i-z_j|^2\right)^{-(n-2s)/2}
\end{align}
and then setting
\begin{align}\label{defn: Q interaction term}
    Q = \max\{q_{ij}:i,j=1,\cdots, \nu\}.
\end{align}
Since the bubbles are $\delta$-interacting, $Q<\delta.$

Step (\rm{iii}). Next, we can estimate the second term $\rho_1$. First observe that $\rho_1$ satisfies
\begin{align}
\left\{\begin{array}{l}
(-\Delta)^s \rho_{1}-\left[\left(\sigma+\rho_{0}+\rho_{1}\right)^{p}-\left(\sigma+\rho_{0}\right)^{p}\right]-\sum_{i=1}^{\nu} \sum_{a=1}^{n+1} c_{a}^{j} U_{j}^{p-1} Z_{j}^{a}-f=0 \\ \innerhs{\rho_1}{Z_j^a}=0, \quad i=1,\cdots,\nu;\quad  a=1, \cdots, n+1\end{array}\right.
\end{align}
Using the above equation and the decomposition 
\begin{align*}
\rho_1 = \sum_{i=1}^{\nu} \beta^{i} U_{i}+\sum_{i=1}^{\nu} \sum_{a=1}^{n+1} \beta_{a}^{i} Z_{i}^{a}+\rho_{2}, 
\end{align*}
we can estimate $\rho_2$ and the absolute value of the scalar coefficients $\beta^i$ and $\beta^i_a$. This allows us to estimate the energy of $\rho_1.$ In particular using similar arguments as in the proofs of Proposition 3.10, 3.11, and 3.12 in \cite{deng2021sharp} one can show that
\begin{align*}
\norml{(-\Delta)^{s/2}\rho_1}\lesssim Q^2+\|f\|_{H^{-s}}.  
\end{align*}

Step (\rm{iv}). In the final step, combining the energy estimates for $\rho_0$ and $\rho_1$ from Step (\rm{ii}) and (\rm{iii}) one can finally arrive at the desired estimate \eqref{fractional-main-estimate} using the same argument as in Lemma 2.1 and Lemma 2.3 in \cite{deng2021sharp}. This will conclude the proof of Theorem \ref{fractional-theorem-main} in the case when the dimension $n\geq 6s.$ 

Even though the proof sketch outlined above seems to carry over in a straightforward manner we decided to include this case for a few reasons. First, we were curious to understand how the fractional parameter $s$ would influence the estimates in our main Theorem. Second, we tried to explain the intuition behind some technical estimates which we hope might be useful for readers not very familiar with the finite-dimensional reduction method (see \cite{finite-dimensional-reduction} for an introduction to this technique). Finally, we encountered an interesting technical issue in Step (\rm{ii}) of the proof. The authors of \cite{deng2021sharp} make use of the following pointwise differential inequality in the case when $s=1$.
\begin{proposition}
The functions $\tilde{W}$ and $\tilde{V}$ satisfy
\begin{align}\label{laplace estimate}
    (-\Delta)^s \tilde{W}\geq \alpha_{n,s} \tilde{V},
\end{align}
where $\alpha_{n,s}>0$ is a constant depending on $n$ and $s.$
\end{proposition}
For precise definitions of $\Tilde{W}$ and $\Tilde{V}$ see \eqref{tilde w and v defn}. After some straightforward reductions of $\Tilde{W}$ and $\Tilde{V}$, it suffices to establish
\begin{align}\label{frac-ineq}
    (-\Delta)^s[(1+|x|^2)^{-s}]\geq \alpha_{n,s} [1+|x|^2]^{-2s}.
\end{align}
for some constant $\alpha_{n,s}>0.$ By direct computations the above inequality is true in the local setting $s=1$ however it is not clear whether it generalizes to the fractional case as well. We show that this is indeed true. 

Such pointwise differential inequality seems to be new and for instance, does not follow from the well-known pointwise differential inequality 
\begin{align*}
(-\Delta)^{s}(\phi(f))(x) \leq \phi^{\prime}(f(x)) \cdot (-\Delta)^{s} f(x)
\end{align*}
where $\phi\in C^1(\R)$ is convex and $f\in \mathcal{S}(\R^n)$. We prove \eqref{frac-ineq} inequality directly by using an integral representation of the fractional derivative of $\frac{1}{(1+|x|^2)^s}$ and by counting the number of zeros of a certain Hypergeometric function.
\section{Case when Dimension $n<6s $}
The goal of this section is to prove Theorem \ref{fractional-theorem-main} in the case when the dimension satisfies $2s<n<6s$. The proof is similar to the proof of Theorem 3.3 in \cite{FigGla20}, however, we need to make some modifications to establish the spectral and interaction integral term estimate. We first prove Theorem \ref{fractional-theorem-main} assuming that the desired spectral estimate holds.
\begin{proof}
Consider the following minimization problem 
\begin{align*}
\normhs{u -\sum_{i=1}^{\nu}\alpha_i  U[z_i, \lambda_i] }= \min_{\bar{z}_{1}, \ldots, \bar{z}_{\nu} \in \mathbb{R}^n \atop \bar{\lambda}_{1}, \ldots, \bar{\lambda}_{\nu}\in \mathbb{R}_{+}, \tilde{\alpha}_1,\ldots,\tilde{\alpha}_{\nu} \in \mathbb{R}}\normhs{ u-\sum_{i=1}^{\nu}\tilde{\alpha_i}  U[\tilde{z}_i, \tilde{\lambda}_i]}
\end{align*}
where $U_i = U[z_i, \lambda_i]$ and $\sigma = \sum_{i=1}^{\nu} \alpha_i U_i$ is the sum of Talenti Bubbles closest to the function $u$ in the $\dot{H}^{s}(\R^n)$ norm. Thus if we set
\begin{align*}
    u = \sigma + \rho
\end{align*}
we immediately have that $\normhs{\rho}\leq \delta$ and since the family $(\tilde{U}_i)_{i=1}^{\nu}$ is $\delta-$interacting we also deduce that the family $(\alpha_i, U_i)_{i=1}^{\nu}$ is $\delta'-$interacting where $\delta'\to 0$ as $\delta \to 0.$ 

Furthermore, for each bubble $U_i$ with $1\leq i \leq \nu$, we have the following orthogonality conditions as before
\begin{align}
\innerhs{\rho}{U_i} &= 0 \label{fractional-bubble-ortho-func}\\
\innerhs{\rho}{\partial_{\lambda}U_i} &= 0 \label{fractional-bubble-ortho-partial-scalar}\\
\innerhs{\rho}{\partial_{z_j}U_i} &= 0\label{fractional-bubble-ortho-partial-z}
\end{align}
for any $1\leq j\leq n.$ Using Lemma \ref{fractional-spectral-lemma}, we deduce that $U, \partial_{\lambda} U$ and $\partial_{z_j} U$ are eigenfunctions for the operator $\frac{(-\Delta)^s}{U^{p-1}}$ and thus for each $1\leq i,j\leq n$, we get
\begin{align}
   \int_{\R^n} \rho \cdot U_i^p &= 0 \label{fractional-rho-Up-ortho-cond},\\
    \int_{\R^n} \rho \cdot  \partial_{\lambda} U_i U_i^{p-1} &=0, \\
    \int_{\R^n} \rho \cdot  \partial_{z_j} U_i U_i^{p-1} &=0.
\end{align}
Next using integration by parts and \eqref{fractional-bubble-ortho-func}, we get
\begin{align*}
    \normhs{\rho}^2=\innerhs{\rho}{\rho} &= \innerhs{\rho}{u-\sigma} =\innerhs{\rho}{u} = \innerl{\rho}{(-\Delta)^s u} \\
    &= \innerl{\rho}{(-\Delta)^s u-u|u|^{p-1}} + \innerl{\rho}{u|u|^{p-1}}\\
    &\leq \normhs{\rho} \normd{(-\Delta)^s u-u|u|^{p-1}} + \int_{\R^n} u|u|^{p-1}\rho.
\end{align*}
In order to estimate the second term, consider the following two estimates
\begin{align*}
    |(\sigma + \rho)|\sigma + \rho|^{p-1} - \sigma|\sigma|^{p-1}| &\leq p|\sigma|^{p-1} |\rho| + C_{n,s}(|\sigma|^{p-2} |\rho|^2 + |\rho|^{p}) \\
    | \sigma|\sigma|^{p-1} - \sum_{i=1}^{\nu} \alpha_i |\alpha_i|^{p-1}U_i^p| &\lesssim \sum_{1\leq i\neq j\leq \nu} U_i^{p-1} U_j,
\end{align*}
where $C_{n,s}>0$ is some constant depending on $n$ and $s.$ Thus using triangle inequality and the orthogonality condition \eqref{fractional-rho-Up-ortho-cond}, we get
\begin{align}\label{fractional-prelim-estimate}
\int_{\R^n} \rho u|u|^{p-1} &\leq  p\int_{\R^n} |\sigma|^{p-1} \rho^2 \nonumber \\
&\quad + C_{n,\nu}\left(\int_{\R^n}|\sigma|^{p-2} |\rho|^3 + \int_{\R^n}|\rho|^{p+1}+\sum_{1\leq i\neq j\leq \nu} \int_{\R^n}\rho U_i^{p-1} U_j\right).
\end{align}
For the first term by Lemma \ref{fractional-spectral-lemma}, we get
\begin{align}\label{fractional-bubble-spectral-estimate}
    p \int_{\R^n} |\sigma|^{p-1} \rho^2 \leq \tilde{c} \normhs{\rho}^2
\end{align}
provided $\delta'$ is small, where $\tilde{c} = \tilde{c}(n,\nu)<1$ is a positive constant. Using Hölder and Sobolev inequality for the remainder terms with the constraint $n<6s$ yields
\begin{align}\label{fractional-easy-terms}
\begin{aligned}
&\int_{\mathbb{R}^{n}} |\sigma|^{p-2}|\rho|^{3} \leq\|\sigma\|_{L^{2^{*}}}^{p-2}\|\rho\|_{L^{2^{*}}}^{3} \lesssim \normhs{\rho} ^{3},\\
&\int_{\mathbb{R}^{n}}|\rho|^{p+1} \lesssim \normhs{\rho}^{p+1} ,\\
&\int_{\mathbb{R}^{n}}\rho U_{i}^{p-1} U_{j} \lesssim \normc{\rho} \normcd{U_i^{p-1}U_j} \lesssim \normhs{\rho} \normcd{U_i^{p-1}U_j}.
\end{aligned}
\end{align}
For the last estimate using an integral estimate similar to Proposition B.2 in Appendix B of \cite{FigGla20}, we get
\begin{align}\label{fractional-dimension-dependant-estimate}
   \norm{U_i^{p-1}U_j}_{L^{2n/(n+2s)}} &= \left(\int_{\R^n} U_i^{2p n/(n+2s)}U_j^{2n/(n+2s)}\right)^{(n+2s)/2n}  \approx Q_{ij}^{(n-2s)/2}\nonumber\\
   & \approx \int_{\R^n} U_i^{p} U_j,
\end{align}
where 
$$Q_{ij} = \min \left(\frac{\lambda_{j}}{\lambda_{i}},\frac{\lambda_{i}}{\lambda_{j}}, \frac{1}{\lambda_{i} \lambda_{j}\left|z_{i}-z_{j}\right|^{2}}\right), i\neq j.
$$
Furthermore, for the interaction term using the same argument as in Proposition 3.11 in \cite{FigGla20},  we have that for any $\epsilon > 0$, there exists a small enough $\delta'>0$ such that
\begin{align}\label{fractional-bubble-interaction-estimate}
    \int_{\R^n} U_i^p U_j \lesssim \epsilon\normhs{\rho}+\normd{(-\Delta)^s u - u|u|^{p-1}}+\normhs{\rho}^{2},
\end{align}
and 
\begin{align}\label{fractional-bubble-coefficient-estimate}
    |\alpha_i - 1| \lesssim \epsilon\normhs{\rho}+\normd{(-\Delta)^s u - u|u|^{p-1}}+\normhs{\rho}^{2}.
\end{align}
Thus using \eqref{fractional-bubble-spectral-estimate}, \eqref{fractional-easy-terms}, \eqref{fractional-dimension-dependant-estimate},  and \eqref{fractional-bubble-interaction-estimate} into \eqref{fractional-prelim-estimate}, we get
\begin{align*}
    \normhs{\rho}^2 \leq (\tilde{c}+ C\epsilon) \normhs{\rho}^2 + C_0(\normhs{\rho}^3 + \normhs{\rho}^{p+1} + \normhs{\rho}\normd{(-\Delta)^s u - u|u|^{p-1}}).
\end{align*}
Therefore if we choose $\epsilon$ such that the quantity $ (\tilde{c}+ C\epsilon) <1$ then we can absorb the term $(\tilde{c}+ C\epsilon) \normhs{\rho}^2$ on the left hand side of the inequality to obtain
\begin{align*}
    \normhs{\rho}^2 \lesssim \normhs{\rho}^3 + \normhs{\rho}^{p+1} + \normhs{\rho}\normd{(-\Delta)^s u - u|u|^{p-1}}.
\end{align*}
Assuming w.l.o.g. that the quantity $\normhs{\rho}\ll 1$, we can now deduce that
\begin{align}\label{fractional-final-rho-estimate}
    \normhs{\rho} \lesssim \normd{(-\Delta)^s u - u|u|^{p-1}}.
\end{align}
Finally, if we decompose $u$ in the following manner
$$u = \sum_{i=1}^{\nu} U_i + \sum_{i=1}^{\nu} (\alpha_i - 1) U_i + \rho = \sum_{i=1}^{\nu} U_i + \rho',$$
then using \eqref{fractional-bubble-coefficient-estimate} and \eqref{fractional-final-rho-estimate}, we get
\begin{align*}
    \normhs{\rho'} \lesssim \sum_{i=1}^{\nu} |\alpha_i-1| + \normhs{\rho} \lesssim \normd{(-\Delta)^s u - u|u|^{p-1}}.
\end{align*}
Thus $U_1, U_2, \cdots, U_{\nu}$ is the desired family of Talenti bubbles.
\end{proof}
We now turn to prove the spectral estimate \eqref{fractional-bubble-spectral-estimate}. For this, we begin by constructing bump functions that will localize the linear combination of bubbles.

\subsection{Construction of Bump Functions}
The construction of bump functions allows us to localize the sum of bubbles $\sigma$ such that in a suitable region we can assume $\sigma \Phi_i \approx \alpha_i U_i$ for some bubble $U_i$ and associated bump function $\Phi_i.$ As a preliminary step we start by constructing cut-off functions with suitable properties.
\begin{lemma}[Construction of Cut-off Function]\label{fractional-cutoff-lemma}
Let $n > 2s.$ Given a point $\overline{x}\in \R^n$ and two radii $0<r< R$, there exists a Lipschitz cut-off function $\varphi=\varphi_{\overline{x}, r, R}:\R^n \to [0,1]$ such that the following holds
\begin{enumerate}
    \item $\varphi \equiv 1$ on $B(\overline{x}, r).$
    \item $\varphi \equiv 0$ outside $B(\overline{x}, R).$
    \item $\int_{\mathbb{R}^n} |(-\Delta)^{s/2} \phi|^{n/s} \lesssim \log(R/r)^{1-n/s}.$
\end{enumerate}
\end{lemma}
\begin{proof}
For simplicity, we set $\overline{x}=0$. Then we define the function $\varphi:\R^n\to [0,1]$ as follows
\begin{align}
\varphi(x):=\left\{\begin{array}{ll}F(r), &\quad |x| \leq r \\ 
F(|x|), &\quad r <|x|\leq R \\ 
F(R), &\quad |x|>R\end{array}\right.
\end{align}
where $F(x)=F(|x|)=\frac{\log (R/|x|)}{\log (R/r)}.$
The function $\varphi$ clearly satisfies the first and second conditions. Next, we estimate the fractional gradient. Recall the formula for the fractional derivative
\begin{align*}
    (-\Delta)^{s/2}\varphi(x) = \int_{\mathbb{R}^n} \frac{\varphi(x)-\varphi(y)}{|x-y|^{n+s}}dy.
\end{align*}
When $|x|>2R$ then since the integral is non-zero when $|y|<R$ we have $|x-y|>\frac{1}{2}|x|$ and therefore \begin{align*}
    |(-\Delta)^{s/2}\varphi(x)|\lesssim \log^{-1}(R/r)\frac{R^n}{|x|^{n+s}}.
\end{align*}
Thus
\begin{align*}
\int_{|x|>2R} |(-\Delta)^{s/2}\varphi(x)|^{n/s} dx \lesssim  \log^{-n/s}(R/r).
\end{align*}
On the other hand when $|x|<2R$ then we have three possible cases. When $|x|<r$ then we consider three cases $r<|y|<2r$ and $2r<|y|<R$ and $|y|>R.$ Since $|x|<r$ and $r<|y|<2r$ implies that $|x-y|<3r$ we have
\begin{align*}
  &\int_{|x|<r} |(-\Delta)^{s/2}\varphi(x)|^{n/s} dx \leq \int_{|x|<r} \left(\int_{|x-y|<3r }\frac{\left|\varphi(x)-\varphi(y)\right|}{|x-y|^{n+s}}dy\right)^{n/s} dx\\
  &\quad +\int_{|x|<r} \left(\int_{2r<|y|<R}\frac{\left|\varphi(x)-\varphi(y)\right|}{|x-y|^{n+s}}dy\right)^{n/s}dx +\int_{|x|<r}\left(\int_{|y|>R}\frac{\left|\varphi(x)-\varphi(y)\right|}{|x-y|^{n+s}}dy\right)^{n/s}dx.
\end{align*}
Denote
\begin{align*}
\mathrm{I} &= \int_{|x-y|<3r }\frac{\left|\varphi(x)-\varphi(y)\right|}{|x-y|^{n+s}}dy,\quad  \mathrm{II} = \int_{2r<|y|<R}\frac{\left|\varphi(x)-\varphi(y)\right|}{|x-y|^{n+s}}dy\\
\mathrm{III} &= \int_{|y|>R}\frac{\left|\varphi(x)-\varphi(y)\right|}{|x-y|^{n+s}}dy.
\end{align*}
Estimating each term
\begin{align*}
\rm{I} &\lesssim \frac{\log^{-1}(R/r)}{r}\int_{|x-y|<3r} \frac{1}{|x-y|^{n+s-1}}dy\lesssim  \frac{\log^{-1}(R/r)}{r^s}  \\
\rm{II} &= \log^{-1}(R/r) \int_{ 2r<|y|<R} \frac{\log(|y|)-\log(r)}{|x-y|^{n+s}}dy \\
&= \log^{-1}(R/r) \sum_{k=1}^{K}\int_{ 2^kr<|y|<2^{k+1}r} \frac{\log(|y|)-\log(r)}{|x-y|^{n+s}}dy\\
&=  \frac{\log^{-1}(R/r)}{r^{n+s}} \sum_{k=1}^{K}\int_{ 2^kr<|y|<2^{k+1}r} \frac{(k+1)\log 2}{(2^k-1)^{n+s}}dy\\
&\lesssim \frac{\log^{-1}(R/r)}{r^s} \sum_{k=1}^{K}\frac{(k+1)2^{kn}}{(2^k-1)^{n+s}}\\
&\lesssim \frac{\log^{-1}(R/r)}{r^s} K \lesssim \frac{1}{r^s}\\
\rm{III} &\lesssim \frac{1}{R^s},
\end{align*}
where we used the fact the $\varphi$ is Lipschitz for estimating $\mathrm{I}$ and $K\approx \log(R/r)$. This implies that
\begin{align*}
    \int_{|x|<r}  |(-\Delta)^{s/2}\varphi(x)|^{n/s} dx &\lesssim \log^{-n/s}(R/r) + 1 + \frac{r^n}{R^n} \\
    &\lesssim \log^{1-n/s}(R/r).
\end{align*}
When $r<|x|<R$ then we can use the expression in Table 1 of \cite{Kwasnicki} to deduce that
\begin{align*}
    |(-\Delta)^{s/2}\varphi(x)| \lesssim \log^{-1}(R/r)\frac{1}{|x|^s}.
\end{align*}
\begin{align*}
\int_{r<|x|<R} |(-\Delta)^{s/2}\varphi(x)|^{n/s} dx \lesssim  \log^{1-n/s}(R/r).
\end{align*}
Following a similar argument as in the case when $|x|<r$ we can estimate in the regime when $R<|x|<2R$ to get 
\begin{align*}
  &\int_{R<|x|<2R} |(-\Delta)^{s/2}\varphi(x)|^{n/s}dx \lesssim  \log^{1-n/s}(R/r).
\end{align*}
Thus combining the above estimates we get
\begin{align*}
    \int_{\mathbb{R}^n} |(-\Delta)^{s/2}\varphi(x)|^{n/s}dx
    &\lesssim \log^{1-n/s}(R/r).
\end{align*}

\end{proof}
The construction of cut-off functions allows us to deduce the following lemma whose proof is identical to the proof of Lemma $3.9$ in \cite{FigGla20}, except for the $L^{n/s}$ estimate which is a consequence of property (\rm{iii}) in Lemma \ref{fractional-cutoff-lemma}.

\begin{lemma}[Construction of Bump Function]\label{fractional-bump-lemma}
Given dimension $n>2s$, number of bubbles $\nu\geq 1$ and parameter $\hat{\epsilon}>0$ there exists a $\delta = \delta(n, \nu, \hat{\epsilon},s) > 0$ such that for a $\delta-$interacting family of Talenti bubbles $(U_i)_{i=1}^{\nu}$ where $U_i = U[z_i, \lambda_i]$ there exists a family of Lipschitz bump functions $\Phi_i:\R^n \to [0,1]$ such that the following hold,
\begin{enumerate}
\item Most of the mass of the function $U_i^{p+1}$ is in the region $\{\Phi_i = 1\}$ or more precisely, 
\begin{align}\label{fractional-bump-lemma-bubble-mass-conc}
\int_{\left\{\Phi_{i}=1\right\}} U_{i}^{p+1} \geq(1-\hat{\epsilon}) S^{n/s}.
\end{align}
\item The function $\Phi_i$ is much larger than any other bubble in the region $\{\Phi_i > 0\}$ or more precisely, 
    \begin{align}\label{fractional-bump-lemma-epsilon U_i}
        \hat{\epsilon} U_{i}>U_{j}
    \end{align}
    for each index $j\neq i.$
    \item The $L^{n/s}$ norm of the the function $(-\Delta)^{s/2} \Phi_{i}$ is small, or more precisely, 
    \begin{align}\label{fractional-bump-lemma-Ln-norm}
        \left\|(-\Delta)^{s/2} \Phi_{i}\right\|_{L^{n/s}} \leq \hat{\epsilon}.
    \end{align}
    \item Finally, for all $j\neq i$ such that $\lambda_j \leq \lambda_i$, we have
    \begin{align}\label{fractional-bump-lemma-sup/inf}
        \frac{\sup _{\left\{\Phi_{i}>0\right\}} U_{j}}{\inf _{\left\{\Phi_{i}>0\right\}} U_{j}} \leq 1+\hat{\epsilon}.
    \end{align}
\end{enumerate}
\end{lemma}

\subsection{Spectral Properties of the Linearized Operator}\label{section:spectral}
Consider the linearized equation,
\begin{align*}
    (-\Delta)^s \phi = pU^{p-1} \phi 
\end{align*}
where $\phi \in \dot{H}^s$ and $U(x)=U[z,\lambda](x).$ By exploiting the positivity of the second variation of $\delta(u) = \normhs{u}^2 - S^2 \normc{u}^2$ around the bubble $U$ and using Theorem 1.1 in \cite{PinoSpectral}, we can deduce the following result
\begin{lemma}\label{fractional-spectral-lemma}
The operator $\mathcal{L}=\frac{(-\Delta)^s}{U^{p-1}}$ has a discrete spectrum with increasing eigenvalues $\{\lambda_{i}\}_{i=1}^{\infty}$ such that,
\begin{enumerate}
    \item The first eigenvalue $\alpha_{1} = 1$ with eigenspace $H_1 = \operatorname{span}(U).$
    \item The second eigenvalue $\alpha_{2} = p$ with eigenspace \newline $H_2 =\operatorname{span}(\partial_{z_1}U,\partial_{z_2}U,\cdots, \partial_{z_n}U,\partial_{\lambda}U)$.
\end{enumerate}
\end{lemma}
\begin{proof}
Since the embedding $\dot{H}^s \hookrightarrow L^{2^*}_{U^{p-1}}$ is compact, the spectrum of the operator $\mathcal{L}$ is discrete. This can be proved by following the same strategy as in the proof of Proposition A.1 in \cite{FigGla20} along with a fractional Rellich-Kondrakov Theorem which is stated as Theorem 7.1 in \cite{rellich}. Furthermore, as the following identities hold
\begin{align*}
    (-\Delta)^s U=U^{p}, \quad(-\Delta)^s\left(\partial_{\lambda} U\right)=p U^{p-1} \partial_{\lambda} U, \quad(-\Delta)^s\left(\nabla_{z_j} U\right)=p U^{p-1} \nabla_{z_j} U,
\end{align*}
for $1\leq j\leq n$, it is clear that $1$ and $p$ are eigenvalues with eigenfunctions $U$ and the partial derivatives $\partial_{z_j}U$ and $\partial_{\lambda} U$ respectively. For the first part, since the function $U>0$, we deduce that the first eigenvalue is $\alpha_1=1$ which is simple and therefore $H_1=\operatorname{span}(U).$ For the second part, recall the min-max characterization of the second eigenvalue
\begin{align*}
\alpha_{2}=\inf \left\{\frac{\innerhs{w}{w}}{\inner{w}{w}_{L^2_{U^{p-1}}}}: \inner{w}{U}_{L^2_{U^{p-1}}} = 0\right\}.
\end{align*}
We first show that $\lambda_2\leq p$. For this consider the second variation of the quantity $\delta(u) = \normhs{u}^2 - S^2 \normc{u}^2$ around the bubble $U$. Since $U$ is an extremizer for the HLS inequality we know that
\begin{align*}
\frac{d^{2}}{d \epsilon^{2}}\bigg|_{\epsilon=0} \delta(U+\epsilon \varphi) \geq 0, \quad \forall \phi \in \dot{H}^s(\R^n).
\end{align*}
Furthermore using $\int_{\R^n} U^{2^*} = S^{n/s}$ for any $\phi \in \dot{H}^s$, we get
\begin{align*}
    \frac{d^{2}}{d \epsilon^{2}}\bigg|_{\epsilon=0} \delta(U+\epsilon \varphi) &= 2\innerhs{\phi}{\phi} - 2S^2(2-2^*) (S^{n/s})^{2/2^* - 2}\left(\int_{\R^n} U^{p}\phi\right)^2 \\
    &\quad - 2S^2 p (S^{n/s})^{2/2^*-1}\int_{\R^n} U^{p-1}\phi^2.
\end{align*}
Thus, when $\inner{\phi}{U}_{L^2_{U^{p-1}}}=\int_{\R^n} \phi U^p = 0$ we can drop the second term to get
\begin{align}\label{Poincare type inequality}
    p \int_{\R^n} U^{p-1}\phi^2 = p \inner{\phi}{\phi}_{L^2_{U^{p-1}}} \leq \innerhs{\phi}{\phi}.
\end{align} 
This implies that $\lambda_2\leq p$ and the equality is attained when $\phi = \partial_{\lambda} U$ or $\phi = \partial_{z_j} U$ for $1\leq j\leq n.$ Thus the second eigenvalue $\lambda_2 = p.$ Finally, we need to argue that $H_2 =\operatorname{span}(\partial_{z_1}U,\partial_{z_2}U,\cdots, \partial_{z_n}U,\partial_{\lambda}U)$. For this we make use of Theorem 1.1 in \cite{PinoSpectral} which states the following
\begin{theorem}\label{PinoTheorem}
Let $n>2s$ and $s\in (0,1)$. Then the solution
$$
U(x)=\alpha_{n, s}\left(\frac{1}{1+|x|^{2}}\right)^{\frac{n-2s}{2}}
$$
of the equation, $(-\Delta)^s U = U^p$ is nondegenerate in the sense that all bounded solutions of equation $(-\Delta)^s \phi = pU^{p-1}\phi$ are linear combinations of the functions
\begin{align*}
    \partial_{\lambda} U, \partial_{z_1} U, \partial_{z_2} U,\cdots, \partial_{z_n} U.
\end{align*}
\end{theorem}
Thus, if we argue that the solutions to the linearized equation $(-\Delta)^s \phi= p U^{p-1}\phi$ are bounded, then we can apply the above theorem to deduce that $H_2 =\operatorname{span}(\partial_{z_1}U,\linebreak\partial_{z_2}U,\cdots, \partial_{z_n}U,\partial_{\lambda}U)$. For this, we use a bootstrap argument. Let $\phi \in \dot{H}^s$ satisfy the equation $(-\Delta)^s \phi = pU^{p-1}\phi$. Then since $pU^{p-1}\phi \in \dot{H}^s$, we get that $(-\Delta)^s \phi \in \dot{H}^s$, which implies that $\phi\in \dot{H}^{3s}.$ Indeed using the definition of $\dot{H}^{3s}$ norm, we have
\begin{align*}
    \norm{u}_{\dot{H}^{3s}}^2 &= \norml{(-\Delta)^{3s/2}\phi}^2 = \norml{(-\Delta)^{s/2}[(-\Delta)^s\phi]}^2 =  \norm{(-\Delta)^s u}_{\dot{H}^{s}}^2 < +\infty.
\end{align*}
However, now since $\phi \in \dot{H}^{3s}$ we can repeat the same argument to get $\phi \in \dot{H}^{5s}.$ Proceeding in this manner we deduce that $\phi \in \dot{H}^{(2k+1)s}$ for any $k\in \N.$ Thus for large enough $k\in \N$, we have $2(2k+1)s > n$ and therefore by Sobolev embedding for fractional spaces (see for instance Theorem 4.47 in \cite{Francoise}) we get that $\phi \in L^{\infty}(\R^n)$. This allows us to use Theorem \ref{PinoTheorem} and thus we deduce that $H_2 =\operatorname{span}(\partial_{z_1}U,\partial_{z_2}U,\cdots, \partial_{z_n}U,\partial_{\lambda}U)$.
\end{proof}
As done in \cite{FigGla20}, by localizing the linear combination of bubbles using bump functions and the spectral properties derived in Lemma \ref{fractional-spectral-lemma}, we can show the following inequality,
\begin{lemma}
Let $n>2s$ and $\nu\geq 1.$ Then there exists a $\delta > 0$ such that if $(\alpha_i, U_i)_{i=1}^{\nu}$ is a family of $\delta-$interacting family of bubbles and $\rho\in \dot{H}^s(\R^n)$ is a function satisfying the orthogonality conditions \eqref{fractional-bubble-ortho-func}, \eqref{fractional-bubble-ortho-partial-scalar} and \eqref{fractional-bubble-ortho-partial-z} where we denote $U_i = U[z_i,\lambda_i]$ then there exists a constant $\tilde{c}<1$ such that following holds
\begin{align*}
    \int_{\R^n} |\sigma|^{p-1}\rho^2 \leq \frac{\tilde{c}}{p}\int_{\R^n} |(-\Delta)^{s/2} \rho|^2 
\end{align*}
where $\sigma = \sum_{i=1}^{\nu} \alpha_i U_i.$
\end{lemma}

\begin{proof}
Let $\epsilon > 0$ be such that there exists a $\delta>0$ and a family of bump functions $(\Phi_i)_{i=1}^{\nu}$ as in Lemma \ref{fractional-bump-lemma}. Then using \eqref{fractional-bump-lemma-epsilon U_i}, we get 
\begin{align*}
\int_{\R^n} |\sigma|^{p-1}\rho^2 &= \int_{\{\sum_{i} \Phi_i \geq 1\}} |\sigma|^{p-1}\rho^2 + \int_{\{\sum_{i} \Phi_i < 1\}} |\sigma|^{p-1}\rho^2 \\
&\leq  \sum_{i=1}^{\nu} \int_{\{\Phi_i > 0\}} |\sigma|^{p-1}\rho^2 + \int_{\{\sum_{i} \Phi_i < 1\}} |\sigma|^{p-1}\rho^2 \\
&\leq (1+o(1))\sum_{i=1}^{\nu} \int_{\R^n} \Phi_{i}^2 U_i^{p-1}\rho^2 + \int_{\{\sum_{i} \Phi_i < 1\}} |\sigma|^{p-1}\rho^2,
\end{align*}
where $o(1)$ denotes a quantity that tends to $0$ as $\delta\to 0.$ We can estimate the second term using Hölder and Sobolev inequality 
\begin{align}\label{fractional-second-term}
\int_{\left\{\sum \Phi_{i}<1\right\}} |\sigma|^{p-1} \rho^{2} &\leq\left(\int_{\left\{\sum \Phi_{i}<1\right\}} |\sigma|^{2^{*}}\right)^{\frac{p-1}{2^{*}}}\|\rho\|_{L^{2^{*}}}^{2} \nonumber\\
&\leq C \left(\sum_{i=1}^{\nu}\int_{\left\{\Phi_{i}<1\right\}}  U_i^{2^{*}}\right)^{\frac{p-1}{2^{*}}}\normhs{\rho}^{2} \nonumber\\
&\leq o(1) \normhs{\rho}^{2}.
\end{align}
To estimate the first term, we first show that
\begin{align}\label{fractional-baby-inequality}
\int_{\mathbb{R}^{n}}\left(\rho \Phi_{i}\right)^{2} U_{i}^{p-1} \leq \frac{1}{\Lambda} \int_{\mathbb{R}^{n}}\left|(-\Delta)^{s/2}\left(\rho \Phi_{i}\right)\right|^{2}+o(1)\normhs{\rho}^{2}
\end{align}
where $\Lambda > p$ is the third eigenvalue of the operator $\frac{(-\Delta)^s}{U_i^p}.$ To prove this estimate we show that $\rho \Phi_i$ almost satisfies the orthogonality conditions \eqref{fractional-bubble-ortho-func}, \eqref{fractional-bubble-ortho-partial-scalar} and \eqref{fractional-bubble-ortho-partial-z}. Let $f:\R^n\to \R$ be a function equal to $U_i, \partial_{\lambda} U_i$ or $\partial_{z_j} U_i$ up to scaling and satisfying the identity $\int_{\R^n} f U_{i}^{p-1} =1.$ Then using  \eqref{fractional-bubble-ortho-func}, \eqref{fractional-bubble-ortho-partial-scalar} and \eqref{fractional-bubble-ortho-partial-z}, we get
\begin{align*}
    \left|\inner{(\rho \Phi_i)}{f}_{L^{2}_{U^{p-1}}}\right|  &= \left|\int_{\R^n} \rho \Phi_i f U_i^{p-1}\right| = \left|\int_{\R^n} \rho  f U_i^{p-1} (1-\Phi_i)\right| \\ &\leq\|\rho\|_{L^{2^*}}\left(\int_{\mathbb{R}^{n}} f^{2} U_{i}^{p-1}\right)^{1/2}\left(\int_{\left\{\Phi_{i}<1\right\}} U_{i}^{2^{*}}\right)^{s/n} \leq o(1)\normhs{\rho}.   
\end{align*}
Using Lemma \ref{fractional-spectral-lemma}, we can now conclude \eqref{fractional-baby-inequality}. We can further estimate the first term in \eqref{fractional-baby-inequality} by using Theorem A.8 in \cite{KeningProductRule}, which states that the remainder term $\mathcal{C}(\rho, \Phi_i)=(-\Delta)^{s/2}(\rho \Phi_i)-\rho(-\Delta)^{s/2}\Phi_i -\Phi_i(-\Delta)^{s/2}\rho$ satisfies the following estimate
\begin{align}\label{frac-error-estimate}
\norml{\mathcal{C}(\rho, \Phi_i}\leq C\norm{(-\Delta)^{s_{1} / 2} \Phi_i}_{L^{p_{1}}}
\norm{(-\Delta)^{s_{2} / 2} \rho}_{L^{p_{2}}}
\end{align}
provided that $s_{1}, s_{2} \in [0, s], s=s_{1}+s_{2}$ and $p_{1}, p_{2} \in(1,+\infty)$ satisfy
$$
\frac{1}{2}=\frac{1}{p_{1}}+\frac{1}{p_{2}}.
$$
Setting $s_1=s, s_2=0, p_1=n/s$ and $p_2=2^* = \frac{2n}{n-2s}$ using \eqref{frac-error-estimate} and \eqref{fractional-bump-lemma-Ln-norm}, we get
\begin{align}\label{fractional-final-error-estimate}
\norml{\mathcal{C}(\rho, \Phi_i}\leq C\norm{(-\Delta)^{s / 2} \Phi_i}_{L^{n/s}}\norm{\rho}_{L^{2^*}}\leq o(1)\normhs{\rho}.
\end{align}
We can estimate the term $\rho(-\Delta)^{s/2}\Phi_i$ using Hölder and Sobolev inequality along with \eqref{fractional-bump-lemma-Ln-norm} as follows
\begin{align*}
    \int_{\R^n}\rho^2 |(-\Delta)^{s/2} \Phi_i|^2 &\leq \normc{\rho}^2 \norm{(-\Delta)^{s/2} \Phi_i}^2_{L^{n/s}} \leq o(1) \normhs{\rho}^2.
\end{align*}
Thus
\begin{align}\label{fractional-prod-d1}
\sum_{i=1}^{\nu}\norml{(-\Delta)^{s/2}(\rho \Phi_i)} &\leq\sum_{i=1}^{\nu} \norml{\mathcal{C}(\rho, \Phi_i)} + \norml{\Phi_i(-\Delta)^{s/2} \rho} + \norml{\rho (-\Delta)^{s/2}\Phi_i}\nonumber\\
&\leq o(1)\normhs{\rho} + \sum_{i=1}^{\nu}\norml{\Phi_i(-\Delta)^{s/2} \rho}.
\end{align}
Since the bump functions have disjoint support by construction, we get
\begin{align}\label{fractional-prod-d2}
    \sum_{i=1}^{\nu} \int_{\R^n}|\Phi_i (-\Delta)^{s/2}\rho|^2 \leq \normhs{\rho}^2.
\end{align}
Therefore using \eqref{fractional-second-term}, \eqref{fractional-baby-inequality} \eqref{fractional-prod-d1} and \eqref{fractional-prod-d2}, we get
\begin{align*}
\int_{\mathbb{R}^{n}} |\sigma|^{p-1} \rho^{2} &\leq (1+o(1))\sum_{i=1}^{\nu} \int_{\R^n} \Phi_{i}^2 U_i^{p-1}\rho^2 + \int_{\{\sum_{i} \Phi_i < 1\}} |\sigma|^{p-1}\rho^2 \\ 
& \leq
(1+o(1))\left(\frac{1}{\Lambda} \sum_{i=1}^{\nu}\int_{\mathbb{R}^{n}}\left|(-\Delta)^{s/2}\left(\rho \Phi_{i}\right)\right|^{2} + o(1)\normhs{\rho}^{2}\right) + 
 o(1) \normhs{\rho}^{2} \\
 &\leq \left(\frac{1}{\Lambda}+o(1)\right) \normhs{\rho}^{2},
\end{align*}
which implies the desired estimate.
\end{proof}
\subsection{Interaction Integral Estimate}
In this section, we will prove \eqref{fractional-bubble-interaction-estimate} and \eqref{fractional-bubble-coefficient-estimate}. 
\begin{lemma}
Let $2s < n < 6s$ and $\nu \geq 1.$ For any $\epsilon>0$ there exits $\delta>0$ such that the following holds. Let $(\alpha_i, U_i)_{i=1}^{\nu}$ be a $\delta-$interacting family, $u = \sum_{i=1}^{\nu} \alpha_i U_i + \rho $ and $\rho$ satisfies \eqref{fractional-bubble-ortho-func}, \eqref{fractional-bubble-ortho-partial-scalar}, \eqref{fractional-bubble-ortho-partial-z} with $\normhs{\rho} \leq 1.$ Then for all $i=1,2,\cdots, \nu$ we have
\begin{align}\label{multi-bubble-alpha}
\left|\alpha_{i}-1\right| \lesssim \epsilon\normhs{\rho}+\left\|(-\Delta)^{s} u-u|u|^{p-1}\right\|_{H^{-s}}+\normhs{\rho}^{2}
\end{align}
and for each $j\neq i$
\begin{align}\label{multi-bubble-interaction}
\int_{\R^n} U_i^p U_j \lesssim \epsilon\normhs{\rho}+\normd{(-\Delta)^s u - u|u|^{p-1}}+\normhs{\rho}^{2},
\end{align}
\end{lemma}
\begin{proof}
We begin by assuming that the bubbles are ordered by $\lambda_i$ in descending order. Thus $U_1$ is the most concentrated bubble and so on. The proof of this Lemma then proceeds by induction on the index $i.$ Assume that the claim holds for all indices $j<i$ where $1\leq i\leq \nu$ and let $U_i$ be the corresponding bubble and let $V= \sum_{j=1,j\neq i }^{\nu} \alpha_j U_j.$ 

For $\epsilon>0$ (in particular $\epsilon=o(1)$) let $\Phi_i$ be the bump function associated to $U_i$ as in Lemma \ref{fractional-bump-lemma}. Then consider the following decomposition
\begin{align*}
\left(\alpha_i-\alpha_i^{p}\right) U_i^{p}-p (\alpha_i U_i)^{p-1} V=& -(-\Delta)^s \rho+((-\Delta)^s u-u|u|^{p-1})-\sum_{i=1}^{\nu} \alpha_{i} U_{i}^{p} \\ 
&+p(\alpha_i U_i)^{p-1}\rho+\left[(\sigma+\rho)|\sigma+\rho|^{p-1}-\sigma^{p}-p \sigma^{p-1} \rho\right] \\ &+\left[p \sigma^{p-1} \rho-p(\alpha_i U_i)^{p-1} \rho\right] \\ &+\left[(\alpha_i U_i+V)^{p}-(\alpha_i U_i)^{p}-p(\alpha_i U_i)^{p-1} V\right].
\end{align*}
The term $(\alpha_i-\alpha_i^p)U_i^p$ allows us to estimate $|\alpha_i-1|$ while the term $U_i^{p-1}V$ will help us to establish the integral estimate. Furthermore, we want to establish a control that is linear in $\norm{(-\Delta)^s u-u|u|^{p-1}}_{H^{-s}}$, therefore, we introduce the laplacian term on the right-hand side. Notice that on the region $\{\Phi_i > 0\}$ we can use \eqref{fractional-bump-lemma-epsilon U_i} to get
\begin{align*}
    \sum_{j\neq i} \alpha_j U_j^p &\leq o(1)U_i^{p-1} V, \\
    |p\sigma^{p-1}\rho - p(\alpha_i U_i)^{p-1} \rho|&\leq o(1) p|\rho|  U_i^{p-1}, \\
    |\alpha_i U_i+V)^{p}-(\alpha_i U_i)^{p}-p(\alpha_i U_i)^{p-1} V| &\leq  o(1) U_i^{p-1} V,\\
    |(\sigma + \rho)|\sigma + \rho|^{p-1} -\sigma^{p} - p\sigma^{p-1}\rho| &\lesssim |\rho|^{p} + U^{p-2}\rho^2
\end{align*}
Thus combining the above estimates we get
\begin{align*}
&|(\alpha_i-\alpha_i^{p}) U_i^{p}-(p \alpha_i^{p-1}+o(1)) U_i^{p-1} V +  (-\Delta)^s \rho-((-\Delta)^s u-u|u|^{p-1})-p(\alpha_i U_i)^{p-1} \rho| \\ 
&\lesssim|\rho|^{p}+ U_i^{p-2}|\rho|^{2}+o(1)(U_i^{p-1}|\rho|).   
\end{align*}
Testing the above estimate with $f\Phi_i$ where $f = U_i$ or $f=\partial_{\lambda} U_i$ and using orthogonality conditions \eqref{fractional-bubble-ortho-func}, \eqref{fractional-bubble-ortho-partial-scalar} and \eqref{fractional-bubble-ortho-partial-z} we get
\begin{align*}
\left|\int_{\R^n}\left[\left(\alpha_i-\alpha_i^{p}\right) U_i^{p}-\left(p \alpha_i^{p-1}+o(1)\right) U_i^{p-1} V\right] f \Phi_i\right|  \lesssim\left|\int_{\R^n} (-\Delta)^{s/2} \rho  (-\Delta)^{s/2}(f \Phi_i)\right| \\ 
+\left|\int_{\R^n}\left((-\Delta)^{s} u-u|u|^{p-1}\right) f \Phi_i\right| +\left|\int_{\R^n} U_i^{p-1} f \rho \Phi_i\right| +\int_{\R^n}|\rho|^{p}|f| \Phi_i\\+o(1)\int_{\R^n} U_i^{p-2}|f||\rho|^{2} \Phi_i+ o(1)\int_{\R^n} U_i^{p-1}|f||\rho| \Phi_i.
\end{align*}
To estimate the first two terms we make use of the Kato-Ponce inequality \eqref{kato}. Thus for the first term, we have 
\begin{align*}
    \left|\int_{\R^n} (-\Delta)^{s/2} \rho  (-\Delta)^{s/2}(f \Phi_i)\right| &\leq  \int_{\R^n} |\mathcal{C}(f, \Phi_i) (-\Delta)^{s/2} \rho  |  + \int_{\R^n} |(-\Delta)^{s/2} \rho  f  (-\Delta)^{s/2}\Phi_i| \\
    &+ \left|\int_{\R^n} |(-\Delta)^{s/2} \rho  (\Phi_i-1) (-\Delta)^{s/2}f \right| 
\end{align*}
where we used made use of the orthogonality conditions \eqref{fractional-bubble-ortho-func} and \eqref{fractional-bubble-ortho-partial-scalar} for the last term and $\mathcal{C}(f,\Phi_i) = (-\Delta)^{s/2}(f\Phi_i)- f  (-\Delta)^{s/2}\Phi_i -\Phi_i (-\Delta)^{s/2}f.$ Then using \eqref{kato}, Hölder and Sobolev inequalities we get
\begin{align*}
    \int_{\R^n} |\mathcal{C}(f, \Phi_i) (-\Delta)^{s/2} \rho  | &\leq \norml{\mathcal{C}(f,\Phi_i)} \normhs{\rho} \leq \normc{f} \norm{(-\Delta)^{s/2} \Phi_i}_{L^{n/s}} \normhs{\rho}\\
    &\lesssim o(1)\normhs{\rho}\\
    \int_{\R^n} |(-\Delta)^{s/2} \rho  f  (-\Delta)^{s/2}\Phi_i| &\leq \normhs{\rho} \normc{f}\norm{(-\Delta)^{s/2}\Phi_i}_{L^{n/s}} \leq o(1)\normhs{\rho} \\
    \left|\int_{\R^n} |(-\Delta)^{s/2} \rho  (\Phi_i-1) (-\Delta)^{s/2}f \right|  &\leq \normhs{\rho} \left(\int_{\{\Phi_i<1\}} |(-\Delta)^{s/2} f|^2 \right)^{1/2} \leq o(1) \normhs{\rho}.
\end{align*}
Combining the above estimates yields
\begin{align*}
     \left|\int_{\R^n} (-\Delta)^{s/2} \rho  (-\Delta)^{s/2}(f \Phi_i)\right| \lesssim o(1)\normhs{\rho}.
\end{align*}
For the second term
\begin{align*}
\left|\int_{\R^n}\left((-\Delta)^s u-u|u|^{p-1}\right) f \Phi_i\right|&\leq \normd{(-\Delta)^su-u|u|^{p-1}} \norml{(-\Delta)^{s/2}(f\Phi_i)} \\
&\lesssim  \normd{(-\Delta)^su-u|u|^{p-1}}
\end{align*}
as $\norml{(-\Delta)^{s/2}(f\Phi_i)} \lesssim 1$ by \eqref{kato}. The other terms can be estimated in the same way as in \cite{FigGla20}. Thus
\begin{align*}
\left|\int_{\R^n} U_i^{p-1} f \rho \Phi_i\right| &= \left|\int_{\R^n} U_i^{p-1} f \rho (\Phi_i-1)\right|\\
&\lesssim \normhs{\rho} \left(\int_{\{\Phi_i<1\}}\left(U_i^{p-1}|f|\right)^{\frac{2^{*}}{p}}\right)^{\frac{p}{2^{*}}} \\
&\lesssim o(1) \normhs{\rho}, \\
\int_{\R^n}|\rho|^{p}|f| \Phi_i &\leq \normhs{\rho}^p\norml{f} \lesssim \normhs{\rho}^p, \\
\int_{\R^n} U_i^{p-2}|f||\rho|^{2} \Phi_i &\leq \normhs{\rho}^2 \norm{U_i^{p-2}f}_{L^{2^*/(p-1)}} \lesssim \normhs{\rho}^2, \\
\int_{\R^n} U_i^{p-1}|f||\rho| \Phi_i &\leq \normhs{\rho} \norm{U_i^{p-1} f}_{L^{2^*/p}} \lesssim \normhs{\rho}.
\end{align*}
Thus combining the above estimates we finally get
\begin{equation}\label{baby-estimate}
\begin{aligned}
\left|\int_{\R^n}\left[\left(\alpha_i-\alpha_i^{p}\right) U_i^{p}-\left(p \alpha_i^{p-1}+o(1)\right) U_i^{p-1} V\right] f \Phi_i\right| \\ \lesssim o(1)\normhs{\rho} + \left\|(-\Delta)^s u-u|u|^{p-1}\right\|_{H^{-s}}+\normhs{\rho}^2.
\end{aligned}
\end{equation}
Now if we split $V = V_1 + V_2$ where $V_1 =\sum_{j<i} \alpha_jU_j$ and $V_2 = \sum_{j>i}\alpha_jU_j$ then we know by our induction hypothesis that our claim holds for all indices $j<i.$ Thus 
\begin{align}\label{induction-estimate}
    \int_{\R^n} U_i^{p-1} V_{1}|f| \Phi_i \lesssim \int U_i^{p} V_{1} \lesssim  o(1)\normhs{\rho} + \left\|(-\Delta)^s u-u|u|^{p-1}\right\|_{H^{-s}}+\normhs{\rho}^2.
\end{align}
Furthermore using \eqref{fractional-bump-lemma-sup/inf} we have
\begin{align*}
    V_2(x)\Phi_i(x) = (1+o(1))\Phi_i(x), \quad \forall x\in \R^n.
\end{align*}
If $\alpha_i = 1$ then we have nothing to prove, otherwise define $\theta = \frac{p\alpha_i^{p-1}V_2(0)}{\alpha_i-\alpha_i^{p}}$ and thus using the previous estimate and \eqref{induction-estimate} we can re-write \eqref{baby-estimate} as
\begin{equation}\label{key-estimate}
\begin{aligned}
\left|\alpha_i-\alpha_i^{p}\right|\left|\int_{\R^n}\left(U_i^{p}-(1+o(1)) \theta U_i^{p-1}\right) f \Phi_i\right| &\lesssim  o(1)\normhs{\rho} + \left\|(-\Delta)^s u-u|u|^{p-1}\right\|_{H^{-s}}\\
&+\normhs{\rho}^2.
\end{aligned}
\end{equation}
Using \eqref{fractional-bump-lemma-bubble-mass-conc} we can expand the integral on the left-side as follows
\begin{align*}
\int_{\R^n}\left(U_i^{p}-(1+o(1)) \theta U_i^{p-1}\right)f \Phi_i &=\int_{\R^n} U_i^{p} f -\theta \int_{\R^n} U_i^{p-1} f+ \int_{\R^n} U_i^{p} f (\Phi_i-1) \\
&\quad + o(1) \theta  \int_{\R^n} U_i^{p-1} f (\Phi_i - 1),
\end{align*}
where the last two terms can be estimated using \eqref{fractional-bump-lemma-bubble-mass-conc} and $|f|\lesssim U_i$
\begin{align*}
\int_{\R^n} U_i^{p} f (\Phi_i-1) \lesssim \int_{\{\Phi_i<1\}} U_i^{p+1} = o(1).
\end{align*}
Thus we have
\begin{align}\label{error-expansion}
\int_{\R^n}\left(U_i^{p}-(1+o(1)) \theta U_i^{p-1}\right) f \Phi_i=\int_{\R^n} U_i^{p} f-\theta \int_{\R^n} U_i^{p-1} f +o(1).
\end{align}
To prove \eqref{multi-bubble-alpha} we need to show that left-side of the \eqref{error-expansion} cannot be too small for $f=U_i$ or $f=\partial_{\lambda} U_i.$ It is enough to check that
\begin{align}\label{ratio-neq}
\frac{\int_{\R^n} U_i^{2^{*}}}{\int_{\R^n} U_i^{p}} \neq \frac{\int_{\R^n} U_i^{p} \partial_{\lambda} U_i}{\int_{\R^n} U_i^{p-1} \partial_{\lambda} U_i}
\end{align}
because otherwise, we would have that
\begin{align*}
\int_{\R^n}\left(U_i^{p}-(1+o(1)) \theta U_i^{p-1}\right) f \Phi_i=o(1),
\end{align*}
which can be made arbitrarily small. To check \eqref{ratio-neq} observe that 
\begin{align*}
    \frac{\int_{\R^n} U_i^{2^{*}}}{\int_{\R^n} U_i^{p}} = \frac{S^{n/s}}{\int_{\R^n} U_i^{p}} > 0
\end{align*}
while,
\begin{align*}
    \frac{\int_{\R^n} U_i^{p} \partial_{\lambda} U_i}{\int_{\R^n} U_i^{p-1} \partial_{\lambda} U_i} = \frac{\frac{1}{p+1}\frac{d}{d\lambda}_{|\lambda=\lambda_i}\int_{\R^n} U[0,\lambda]^{p+1}}{\frac{1}{p}\frac{d}{d\lambda}_{|\lambda=\lambda_i}\int_{\R^n} U[0,\lambda]^{p}} = 0.
\end{align*}
Thus combining \eqref{error-expansion} and \eqref{ratio-neq} we get
\begin{align*}
   1\lesssim \max_{f\in \{U_i, \partial_{\lambda} U_i\}} \left|\int\left(U_i^{p}-(1+o(1)) \theta U_i^{p-1}\right) f \Phi_i\right|  
\end{align*}
for either $f=U_i$ or $f=\partial_{\lambda} U_i$. Choosing $f$ for which the above integral is maximized we get
\begin{align*}
   \left|\alpha_i-\alpha_i^{p}\right|&\lesssim  \left|\alpha_i-\alpha_i^{p}\right| \max_{f\in \{U_i, \partial_{\lambda} U_i\}} \left|\int\left(U_i^{p}-(1+o(1)) \theta U_i^{p-1}\right) f \Phi_i\right| \\
   &\lesssim o(1)\normhs{\rho} + \left\|(-\Delta)^s u-u|u|^{p-1}\right\|_{H^{-s}}+\normhs{\rho}^2.
\end{align*}
This proves \eqref{multi-bubble-alpha}. To prove \eqref{multi-bubble-interaction} we use \eqref{multi-bubble-alpha} and \eqref{baby-estimate} with $f=U_i$ to get
\begin{align*}
    \int_{\R^n} U_i^p V \Phi_i \lesssim o(1)\normhs{\rho} + \left\|(-\Delta)^s u-u|u|^{p-1}\right\|_{H^{-s}}+\normhs{\rho}^2,
\end{align*}
which in particular implies that for all indices $j\neq i$, we have
\begin{align*}
    \int_{B(z_i, \lambda_i^{-1})} U_i^p U_j \lesssim o(1)\normhs{\rho} + \left\|(-\Delta)^s u-u|u|^{p-1}\right\|_{H^{-s}}+\normhs{\rho}^2.
\end{align*}
Using the integral estimate similar to Proposition B.2 in Appendix B of \cite{FigGla20} we deduce that the \eqref{multi-bubble-interaction} also holds for all indices $j>i$ and thus we are done.
\end{proof}
\section{Case when dimension $n\geq 6s $}The goal of this section is to prove Theorem \ref{fractional-theorem-main} in the case when the dimension satisfies $n\geq 6s$. We follow the steps outlined in Section \ref{section:proof-sketch}.
\subsection{Existence of the first approximation $\rho_0$} In this section, our goal is to find a function $\rho_0$ and a set of scalars $\{c^i_a\}$ such that the following system is satisfied
\begin{align}\label{eqn: rho0 PDE}
\left\{\begin{array}{l}(-\Delta)^s \varphi-\left[\left(\sigma+\varphi\right)^{p}-\sigma^{p}\right]=h-\sum_{i=1}^{\nu} \sum_{a=1}^{n+1} c_{a}^{i} U_{i}^{p-1} Z_{i}^{a} \\ \innerhs{\rho}{Z_i^a}=0, \quad i=1,\cdots,\nu;\quad a=1, \cdots, n+1\end{array}\right.
\end{align}
Here $U_i = U[z_i,\lambda_i]$ are a family of $\delta-$ interacting bubbles, $\sigma = \sum_{i=1}^{\nu} U_i$, $Z_i^{a} $ are derivatives as defined in \eqref{deng-z-derivative-notation}, $h = \sigma^{p}-\sum_{j=1}^{\nu} U_{j}^{p}$. From the finite-dimensional reduction argument we first strive to solve the equation with the linear operator
\begin{align}\label{first approx system}
\left\{\begin{array}{l}(-\Delta)^s \varphi-p\sigma^{p-1}\varphi=h-\sum_{i=1}^{\nu} \sum_{a=1}^{n+1} c_{a}^{i} U_{i}^{p-1} Z_{i}^{a} \\ \innerhs{\rho}{Z_i^a}=0, \quad i=1,\cdots,\nu;\quad a=1, \cdots, n+1\end{array}\right.
\end{align}
Since the proof follows from a fixed point argument, we need to define the appropriate normed space. Denote for $i\neq j$ and $i,j\in I$ where $I=\{1,\cdots, \nu\}$
\begin{align}
R_{ij} &= \max\left(\sqrt{\lambda_{i} / \lambda_{j}}, \sqrt{\lambda_{j} / \lambda_{i}}, \sqrt{\lambda_{i} \lambda_{j}}\left|z_{i}-z_{j}\right|\right) = \varepsilon^{-1}_{ij},\label{R_ij definition}\\
R &= \frac{1}{2}\min_{i\neq j} R_{ij} \label{R definition}.
\end{align}
The quantity $R_{ij}$ gives us an idea of how concentrated or how far the two bubbles are with respect to each other. We further identify two regimes of interest in the subsequent definition. 
\begin{definition}
If $R_{ij} = \sqrt{\lambda_{i} \lambda_{j}}\left|z_{i}-z_{j}\right|$ then we call the bubbles $U_i$ and $U_j$ a bubble cluster. Otherwise, we call them a bubble tower. 
\end{definition}
We now define two norms that capture the behavior of the interaction term $h$.
\begin{definition}
Define the norm $\norm{\cdot}_{*}$ as
\begin{align}\label{deng-star-norm-defn}
  \norm{\phi}_{*} = \sup _{x \in \mathbb{R}^{n}}|\phi(x)| W^{-1}(x)
\end{align}
and the norm $\norm{\cdot}_{**}$ as 
\begin{align}\label{deng-double-star-norm-defn}
  \norm{h}_{**} =  \sup _{x \in \mathbb{R}^{n}}|h(x)| V^{-1}(x),
\end{align}
where the functions $V$ and $W$ are defined using $y_i = \lambda_i(x-z_i)$
\begin{align}
V(x) = \sum_{i=1}^{\nu}\left(\frac{\lambda_{i}^{\frac{n+2s}{2}} R^{2s-n}}{\left\langle y_{i}\right\rangle^{4s}} \chi_{\left\{\left|y_{i}\right| \leq R\right\}}+\frac{\lambda_{i}^{\frac{n+2s}{2}} R^{-4s}}{\left\langle y_{i}\right\rangle^{n-2s}} \chi_{\left\{\left|y_{i}\right| \geq R / 2\right\}}\right) \label{deng-V-func-defn}\\
W(x)=\sum_{i=1}^{\nu}\left(\frac{\lambda_{i}^{\frac{n-2s}{2}} R^{2s-n}}{\left\langle y_{i}\right\rangle^{2s}} \chi_{\left\{\left|y_{i}\right| \leq R\right\}}+\frac{\lambda_{i}^{\frac{n-2s}{2}} R^{-4s}}{\left\langle y_{i}\right\rangle^{n-4s}} \chi_{\left\{\left|y_{i}\right| \geq R / 2\right\}}\right)\label{deng-W-func-defn}.
\end{align}
\end{definition}
Using the norm $\normds{\cdot}$ we can obtain control on the interaction term $h$ for small enough $\delta$. This is the content of the next lemma.
\begin{lemma}\label{lemma h bound in double star}
There exists a small constant $\delta_0 =\delta_{0}(n)$ and large constant $C=C(n)$ such that if $\delta<\delta_0$ then
\begin{align*}
    \normds{\sigma^p - \sum_{i=1}^{\nu}U_i^p} \leq C(n).
\end{align*}
\end{lemma}
\begin{proof}
We first prove the desired estimate when $\nu=2$. Then since 
\begin{align*}
\left|\sigma^p - \sum_{i=1}^{\nu} U_i^p\right| \leq C \sum_{i\neq j}\left[(U_i+U_j)^p -U_i^p-U_j^p\right]
\end{align*}
we can conclude the proof in the case when $\nu>2.$ Thus for the remainder of the proof assume that $\nu = 2.$ As the bubbles are weakly interacting, the bubbles $U_1$ and $U_2$ either form a bubble tower or a bubble cluster. We will analyze each case separately.

Case 1. (Bubble Tower) Assume w.l.o.g. that $\lambda_1 > \lambda_2$, in other words $U_1$ is more concentrated than $U_2.$ Then $R_{12} = \sqrt{\frac{\lambda_1}{\lambda_2}} \gg 1.$ We will estimate the function $h =(U_1+U_2)^p - U_1^p-U_2^p$ in different regimes. 

Core region of $U_1$: We define the core region of $U_1$ as the following set
\begin{align}\label{core_u1}
    \operatorname{Core}(U_1) = \left\{x\in \R^n: |x-z_1| < \frac{1}{\sqrt{\lambda_1 \lambda_2}}\right\}
\end{align}
Working in the $z_1$ centered co-ordinates $y_1 = \lambda_1(x-z_1)$, we first have that
\begin{align*}
    U_1(x) = U[z_1,\lambda_1](x) = \lambda_1^{(n-2s)/2}U[0,1](y_1) = \lambda_1^{(n-2s)/2}U(y_1).
\end{align*}
Furthermore denoting $\lambda_2 = \lambda_1 R_{12}^{-2}$ and $\xi_2 = \lambda_1(z_2-z_1)$ we have the following identity
\begin{align*}
    \lambda_2(x-z_2) = R_{12}^{-2}\lambda_1 (x-z_1) + R_{12}^{-2}\lambda_1(z_1-z_2) = R_{12}^{-2}(y_1 -\xi_2).
\end{align*}
Thus we can express $U_2$ as follows
\begin{align*}
    U_2(x) &= U[z_2,\lambda_2](x) = c_{n,s} \left(\frac{\lambda_2}{1+\lambda_2^{2}|x-z_2|^{2}}\right)^{\frac{n-2s}{2}} \\
    &= c_{n,s} \left(\frac{\lambda_1 R_{12}^{-2}}{1+R_{12}^{-4s}|y_1-\xi_2|^{2}}\right)^{\frac{n-2s}{2}}  =  \frac{c_{n,s}\lambda_1^{(n-2s)/2} R_{12}^{2s-n}} {(1+R_{12}^{-4}|y_1-\xi_2|^{2})^{(n-2s)/2}}.
\end{align*}
Note that by definition of $R_{12}$ in \eqref{R_ij definition}, we have 
\begin{align*}
|\xi_2| = |\lambda_1(z_1-z_2)|\leq \frac{R_{12}\lambda_1}{\sqrt{\lambda_1 \lambda_2}}\leq R_{12}^2.
\end{align*}
Now we are in position to estimate $h$ in the core region of $U_1.$ Observe that for $x\in \operatorname{Core}(U_1)$ we have
\begin{align*}
    |y_1| = |\lambda_1(x-z_1)| < \frac{\lambda_1}{\sqrt{\lambda_1 \lambda_2}} = R_{12}
\end{align*}
and thus if $|y_1|\leq \frac{R_{12}}{2}$ we have that $U_2\lesssim U_1$. Therefore
\begin{align}\label{case_1: core U1}
    h \lesssim U_1^{p-1}U_2 &\approx \frac{\lambda_1^{(n+2s)/2} R_{12}^{2s-n}}{(1+|y_1|^2)^{(n+2s)/2}(1+R_{12}^{-4}|y_1-\xi_2|^{2})^{(n-2s)/2}}\nonumber\\
    &\lesssim \frac{\lambda_1^{(n+2s)/2} R_{12}^{2s-n}}{\langle y_1 \rangle^{n+2s} (1 + |y_1|^{-2})^{(n-2s)/2}} \lesssim \frac{\lambda_1^{(n+2s)/2} R_{12}^{2s-n}}{\langle y_1 \rangle^{4s} }.
\end{align}
Outside the Core region of $U_1$: In this case we consider $R_{12}/3 \leq  |y_1| \leq 2 R_{12}^2$ and thus we get $U_1  \approx \lambda_1^{(n-2s)/2} |y_1|^{2s-n}$. This is because 
\begin{align*}
    U_1 &= c_{n,s} \lambda_1^{(n-2s)/2} \frac{1}{(1+|y_1|^2)^{(n-2s)/2}} \lesssim \lambda_1^{(n-2s)/2} |y_1|^{2s-n},  \\
    \lambda_1^{(n-2s)/2} |y_1|^{2s-n}&\lesssim \frac{\lambda_1^{(n-2s)/2}}{(9|y_1|^2+|y_1|^2)^{(n-2s)/2}} \lesssim U_1 = c_{n,s}  \frac{\lambda_1^{(n-2s)/2}}{(1+|y_1|^2)^{(n-2s)/2}},
\end{align*}
where for the second estimate we used the fact that $\frac{1}{3} \leq \frac{R_{12}}{3} \leq |y_1|$ implies $9|y_1|^2 \geq 1.$ On the other hand, $U_2 \approx \lambda_1^{(n-2s)/2} R_{12}^{2s-n}.$ This is because
\begin{align*}
    U_2 &= \frac{c_{n,s}\lambda_1^{(n-2s)/2} R_{12}^{2s-n}} {(1+R_{12}^{-4}|y_1-\xi_2|^{2})^{(n-2s)/2}} \lesssim \lambda_1^{(n-2s)/2} R_{12}^{2s-n}, \\
     \lambda_1^{(n-2s)/2} R_{12}^{2s-n} &\lesssim U_2 = \frac{c_{n,s}\lambda_1^{(n-2s)/2} R_{12}^{2s-n}} {(1+R_{12}^{-4}|y_1-\xi_2|^{2})^{(n-2s)/2}},
\end{align*}
where for the second estimate we used the fact that $|\xi_2|^2\leq R_{12}^4$ and $|y_1|^2\leq 4R_{12}^4$ to get
\begin{align*}
R_{12}^{-4}|y_1-\xi_2|^2 \leq R_{12}^{-4}(|y_1|^2 + |\xi_2|^2) \leq 5 \lesssim 1.
\end{align*}
Thus we get
\begin{align}\label{case 1: outside core u1}
    h  &= U_2^p \left[(1+U_1/U_2)^p - 1 - (U_1/U_2)^p\right] \nonumber\\&\lesssim \lambda_{1}^{(n+2s) / 2} R_{12}^{-(n+2s)}\left|\left(1+\frac{R_{12}^{n-2s}}{|y_1|^{n-2s}}\right)^{p}-1-\left(\frac{R_{12}}{|y_1|}\right)^{p(n-2s)}\right| \nonumber\\ 
    &\lesssim \frac{\lambda_{1}^{(n+2s) / 2} R_{12}^{-(n+2s)}}{|y_1/R_{12}|^{n-2s}} \approx \frac{\lambda_{1}^{(n+2s) / 2} R_{12}^{-4s}}{\left\langle y_{1}\right\rangle^{n-2s}}.
\end{align}

Core region of $U_2$: We define the core region of $U_2$ as
\begin{align}\label{core_u2}
    \operatorname{Core}(U_2) = \left\{x\in \R^n: |x-z_2| < \frac{1}{\lambda_2} \sqrt{\frac{\lambda_1}{\lambda_2}}\right\}.
\end{align}
Note that in the $z_2$ centered co-ordinates $y_2 = \lambda_2(x-z_2)$ the points $x\in \R^n$ satisfying $|y_2|< R_{12}$ are precisely the ones forming $\operatorname{Core}(U_2).$ Like before, in these new co-ordinates we can re-write $U_2(x) = \lambda_2^{(n-2s)/2} U(y_2)$ and $U_1$ as
\begin{align*}
U_{1}(x)=\frac{\lambda_{2}^{(n-2s) / 2} R_{12}^{2s-n}}{\left(R_{12}^{-4}+\left|y_{2}-\xi_{1}\right|^{2}\right)^{(n-2s) / 2}}
\end{align*}
where $\xi_1 = \lambda_2(z_1-z_2)$ such that 
\begin{align*}
    |\xi_1|=\lambda_2|z_1-z_2|\leq \sqrt{\frac{\lambda_2}{\lambda_1}} \sqrt{\lambda_1\lambda_2}|z_1-z_2|\leq  \sqrt{\frac{\lambda_2}{\lambda_1}}\sqrt{\frac{\lambda_1}{\lambda_2}}=1.
\end{align*}
Since $y_2 - \xi_1 = R_{12}^{-2}y_1$, then in the region $1\leq |y_2-\xi_1| \leq R_{12}/2$ we have $R_{12}^2 \leq |y_1| \leq R_{12}^3/2.$ This is indeed the core region of $U_2$ as $$|y_2|\leq |\xi_1| + R_{12}^{-2}|y_1|\leq  1+R_{12}/2\leq R_{12}$$
and thus $U_1\lesssim U_2$ which implies that
\begin{align}\label{case 1: core of u2}
    h \lesssim U_2^{p-1} U_1 \lesssim \frac{\lambda_2^{(n+2s)/2}R_{12}^{2s-n}}{\langle y_2 \rangle^{4s}}.
\end{align}

Outside the Core region $U_2$: In this region $|y_2-\xi_2| \geq R_{12}/3 $ and, therefore
\begin{align}\label{case 1: outside core of u2}
h \lesssim U_{2}^{p} \lesssim \frac{\lambda_{1}^{(n+2s) / 2} R_{12}^{-2s}}{\left\langle y_{1}\right\rangle^{n}} \lesssim \frac{\lambda_{1}^{(n+2s) / 2} R_{12}^{-4s}}{\left\langle y_{1}\right\rangle^{n-2s}}.
\end{align}

Thus if we put together estimates \eqref{case_1: core U1}, \eqref{case 1: outside core u1}, \eqref{case 1: core of u2} and \eqref{case 1: outside core of u2} we get
\begin{align}\label{h key estimate}
    h \lesssim \sum_{i=1}^{2}\left(\frac{\lambda_{i}^{\frac{n+2s}{2}} R^{2s-n}}{\left\langle y_{i}\right\rangle^{4s}} \chi_{\left\{\left|y_{i}\right| \leq R / 2\right\}}+\frac{\lambda_{i}^{\frac{n+2s}{2}} R^{-4s}}{\left\langle y_{i}\right\rangle^{n-2s}} \chi_{\left\{\left|y_{i}\right| \geq R / 3\right\}}\right),
\end{align}
which in particular implies that $h(x)\leq V(x) C(n,s)$ for some positive dimension and $s$ dependent constant $C(n,s)>0$. Thus, if the bubbles form a tower then Lemma \ref{lemma h bound in double star} holds. The proof when the bubbles form a cluster also follows from the same argument as in the proof of Proposition 3.3 in \cite{deng2021sharp} with the minor modifications in the exponents involving the parameter $s.$
\end{proof}
Next, we look at a system similar to \eqref{first approx system} and deduce an a priori estimate on its solution.

\begin{lemma}\label{first approx: baby existence lemma}
There exists a positive $\delta_{0}$ and a constant $C$, independent of $\delta$, such that for all $\delta \leqslant \delta_{0}$, if $\left\{U_{i}\right\}_{1 \leq i \leq \nu}$ is a $\delta$ -interacting bubble family and $\phi$ solves the equation
\begin{align}\label{first approx: baby pde}
    \left\{\begin{array}{l}
(-\Delta)^s \phi-p \sigma^{p-1} \phi+h=0 \\
\innerhs{\phi}{Z_i^a}=0, \quad i=1, \cdots, \nu ; a=1, \cdots, n+1
\end{array}\right.
\end{align}
then
\begin{align}\label{phi leq h}
\|\phi \|_{*} \leq C\|h\|_{* *},
\end{align}
where the norms $\|\phi \|_{*} = \sup_{x\in \R^n} |\phi(x)|W^{-1}(x)$ and $\normds{h} = \sup_{x\in \R^n}|h(x)|V^{-1}(x)$ with weight functions
\begin{align*}
    V(x) &= \sum_{i=1}^{\nu}\left(\frac{\lambda_{i}^{\frac{n+2s}{2}} R^{2s-n}}{\left\langle y_{i}\right\rangle^{4s}} \chi_{\left\{\left|y_{i}\right| \leq R\right\}}+\frac{\lambda_{i}^{\frac{n+2s}{2}} R^{-4s}}{\left\langle y_{i}\right\rangle^{n-2s}} \chi_{\left\{\left|y_{i}\right| \geq R / 2\right\}}\right),\\
    W(x) &= \sum_{i=1}^{\nu}\left(\frac{\lambda_{i}^{\frac{n-2s}{2}} R^{2s-n}}{\left\langle y_{i}\right\rangle^{2s}} \chi_{\left\{\left|y_{i}\right| \leq R\right\}}+\frac{\lambda_{i}^{\frac{n-2s}{2}} R^{-4s}}{\left\langle y_{i}\right\rangle^{n-4s}} \chi_{\left\{\left|y_{i}\right| \geq R / 2\right\}}\right).
\end{align*}
\end{lemma}
\begin{proof}
We will prove this result following the same contradiction-based argument outlined in \cite{deng2021sharp}. Thus, assume that the estimate \eqref{phi leq h} does not hold. Then there exists sequences $(U^{(k)}_i)_{i=1}^{\nu}$ of $\frac{1}{k}-$interacting family of bubbles, $h = h_k$ with $\normds{h_k}\to 0$ as $k\to \infty$ and $\phi = \phi_k$ such that $\normss{\phi_k} = 1$ and satisfies
\begin{align}
\left\{\begin{array}{l}
(-\Delta)^s \phi_k-p \sigma^{p-1} \phi_k+h_k=0 \\
\innerhs{\phi_k}{Z_i^{a^{(k)}}}, \quad i=1,\cdots,\nu; a=1, \cdots, n+1
\end{array}\right.
\end{align}
Here $Z_i^{a^{(k)}}$ are the partial derivatives of $U_i^{(k)}$ defined as in \eqref{deng-z-derivative-notation}. We will often drop the superscript for convenience and set $U_i = U_i^{(k)}$ and $\sigma = \sigma^{(k)} =\sum_{i=1}^{\nu} U_i^{(k)}.$ Next denote, $z_{ij}^{(k)} = \lambda_i^{(k)}(z_j^{(k)}-z_i^{(k)})$ for $i\neq j.$ For each such sequence we can assume that either $\lim_{k\to \infty} z_{ij}^{(k)} = \infty$ or $\lim_{k\to \infty} z_{ij}^{(k)}$ exists and $\lim_{k\to \infty} |z_{ij}^{(k)}|<\infty$. Thus define the following two index sets for $i\in I=\{1,2,\cdots,\nu\}$
\begin{align*}
I_{i, 1}&= \left\{j\in I\setminus \{i\} | \lim _{k \rightarrow \infty} z_{i j}^{(k)} \text{ exists and } \lim _{k \rightarrow \infty}|z_{i j}^{(k)}|<\infty\right\},\\
I_{i, 2}&=\left\{j\in I\setminus \{i\} | \lim _{k \rightarrow \infty} z_{i j}^{(k)}=\infty\right\}.
\end{align*}
Given a bubble $U_i$ we can further segregate it from bubble $U_j$, for $j\neq i$, based on whether they form a bubble cluster or bubble tower with respect to each other
\begin{align*}
T_{i}^{+} &=\left\{j\in I\setminus \{i\} \mid \lambda_{j}^{(k)} \geq \lambda_{i}^{(k)}, R_{ij}^{(k)}=\sqrt{\lambda_{j}^{(k)} / \lambda_{i}^{(k)}}\right\} \\ 
T_{i}^{-} &=\left\{j\in I\setminus \{i\} \mid \lambda_{j}^{(k)}<\lambda_{i}^{(k)}, R_{ij}^{(k)}=\sqrt{\lambda_{i}^{(k)} / \lambda_{j}^{(k)}}\right\} \\ 
C_{i}^{+} &=\left\{j\in I\setminus \{i\} |\lim _{k \rightarrow \infty} \lambda_{j}^{(k)} / \lambda_{i}^{(k)}=\infty, R_{ij}^{(k)}=\sqrt{\lambda_{i}^{(k)} \lambda_{j}^{(k)}}| z_{i}^{(k)}-z_{j}^{(k)}|\right\} \\ 
C_{i}^{-} &=\left\{j\in I\setminus \{i\}|\lim _{k \rightarrow \infty} \lambda_{j}^{(k)} / \lambda_{i}^{(k)}<\infty, R_{ij}^{(k)}=\sqrt{\lambda_{i}^{(k)} \lambda_{j}^{(k)}}| z_{i}^{(k)}-z_{j}^{(k)}|\right\} .
\end{align*}
Setting $y_i^{(k)} = \lambda_i^{(k)}(x-z_i^{(k)})$ and $L_i = \max_{j\in I_{i,1}}\{\lim_{k\to \infty} |z_{ij}^{(k)}|\}+L$ we define
\begin{align}
    \Omega^{(k)} &= \bigcup_{i\in I} \{|y_i^{(k)}|\leq L_i\}, \\
    \Omega_{i}^{(k)} &=\left\{|y_{i}^{(k)}| \leq L_i\right\} \bigcap\left(\bigcap_{j \in T_{i}^{+} \cup C_{i}^{+}}\left\{\left|y_{i}^{(k)}-z_{i j}^{(k)}\right| \geq \epsilon\right\}\right) \label{omega defn}
\end{align}
where $i\in I$, $j\neq i$, $L>0$ is a large constant and $\epsilon>0$ is a small constant to be determined later. Observe that for large $k$, our choice of $L_i$ ensures that $\{|y_i^{(k)}-z_{ij}^{(k)}|=\epsilon\}\subset \{|y_i^{(k)}|\leq L_i\}$ for $j\in T_i^+\cup C_i^+.$

Dropping superscripts we re-write the weight functions $W$ and $V$ as defined earlier in the following manner
\begin{align*}
    W = \sum_{i=1}^{\nu} w_{i,1} + w_{i,2},\quad V = \sum_{i=1}^{\nu} v_{i,1} + v_{i,2},
\end{align*}
where for each $i\in I$ and $R=\frac{1}{2}\min_{j\neq i} R_{ij}^{(k)}\to \infty$ as $k\to \infty$, we have
\begin{align}\label{w and v defn}
w_{i, 1}(x) &=\frac{\lambda_{i}^{\frac{n-2s}{2}} R^{2s-n}}{\left\langle y_{i}\right\rangle^{2s}} \chi_{\left\{\left|y_{i}\right| \leq R\right\}}, \quad w_{i, 2}(x)=\frac{\lambda_{i}^{\frac{n-2s}{2}} R^{-4s}}{\left\langle y_{i}\right\rangle^{n-4s}} \chi_{\left\{\left|y_{i}\right| \geq R / 2\right\}}, \\ v_{i, 1}(x) &=\frac{\lambda_{i}^{\frac{n+2s}{2}} R^{2s-n}}{\left\langle y_{i}\right\rangle^{4s}} \chi_{\left\{\left|y_{i}\right| \leq R\right\}}, \quad v_{i, 2}(x)=\frac{\lambda_{i}^{\frac{n+2s}{2}} R^{-4s}}{\left\langle y_{i}\right\rangle^{n-2s}} \chi_{\left\{\left|y_{i}\right| \geq R / 2\right\}}.
\end{align}
Since $\normss{\phi_k}=1$, by definition we have that $|\phi_k|(x) \leq W(x)$. In particular, there exists a subsequence such that $|\phi_k|(x_k)=W(x_k).$ We will now analyze several cases based on where this subsequence $\{x_k\}_{k \in \N}$ lives. 

Case 1. Suppose $\{x_k\}_{k \in \N}\in \Omega_i$ for some $i\in I.$ W.l.o.g. $i=1$ and set
\begin{align}\label{tilde phi h sigma defn}
\left\{\begin{array}{l}\tilde{\phi}_{k}\left(y_{1}\right):=W^{-1}\left(x_{k}\right) \phi_{k}(x) \quad \text { with } y_{1}=\lambda_{1}\left(x-z_{1}\right) \\ \tilde{h}_{k}\left(y_{1}\right):=\lambda_{1}^{-2s} W^{-1}\left(x_{k}\right) h_{k}(x), \quad \tilde{\sigma}_{k}\left(y_{1}\right):=\sigma_{k}(x)\end{array}\right.
\end{align}
Then since $\phi_k$ solves the system \eqref{first approx: baby pde}, the rescaled function $\tilde{\phi}_k$ and $\tilde{h}_k$ satisfy
\begin{align}\label{tilde phi_k system}
\left\{\begin{array}{ll}
(-\Delta)^s \tilde{\phi}_{k}\left(y_{1}\right)-p \lambda_{1}^{-2s} \tilde{\sigma}_{k}^{p-1}\left(y_{1}\right) \tilde{\phi}_{k}\left(y_{1}\right)+\tilde{h}_{k}\left(y_{1}\right)=0 & \text { in } \mathbb{R}^{n} \\ \int_{\R^n} U^{p-1} Z_{1}^{a} \tilde{\phi}_{k} d y_{1}=0, & 1 \leq a \leq n+1\end{array}\right.
\end{align}
where $U = U[0,1](y_1)$. Set, $\tilde{z}_j = z_{1j}^{(k)}$, $\bar{z}_j = \lim_{k\to \infty} \tilde{z}_{1j}^{(k)}$ and define
\begin{align}\label{Set Km defn}
E_{1} &=\bigcap_{j \in T_{1}^{+} \cap I_{1,1}}\left\{\left|y_{1}-\tilde{z}_{j}\right| \geq 1 / M\right\}, \quad E_{2}=\bigcap_{j \in C_{1}^{+} \cap I_{1,1}}\left\{\left|y_{1}-\bar{z}_{j}\right| \geq\left|\bar{z}_{j}\right| / M\right\}, \nonumber\\
K_M &=\left\{\left|y_{1}\right| \leq M\right\} \cap E_{1} \cap E_{2}.
\end{align}
If we choose $M$ to be large enough then $\Omega_1^{(k)}\ssubset K_M$ for $k$ large enough. To prove estimate \eqref{phi leq h} in Case 1 we need to make use of the following proposition.

\begin{proposition}\label{Proposition 1}
In each compact subset $K_{M}$, it holds that, as $k \rightarrow \infty$
\begin{align}\label{Proposition 1: convergence}
\lambda_{1}^{-2s} \tilde{\sigma}_{k}^{p-1} \rightarrow U[0,1], \quad\left|\tilde{h}_{k}\right| \rightarrow 0
\end{align}
uniformly. Moreover, we have
\begin{align}\label{Proposition 1: estimate}
\left|\tilde{\phi}_{k}\right|\left(y_{1}\right) \lesssim\left|y_{1}-\tilde{z}_{j}\right|^{4s-n}+1, \quad j \in T_{1}^{+} \cup C_{1}^{+}.
\end{align}
\end{proposition}

Postponing the proof of Proposition \ref{Proposition 1} we strive to prove \eqref{phi leq h}. First note that by elliptic regularity theory and \eqref{Proposition 1: convergence} we get that $\tilde{\phi}_k$ converges upto a subsequence in $K_M$. When $M\to \infty$, using a diagonal argument we get that $\tilde{\phi}_k\to \tilde{\phi}$ weakly such that $\tilde{\phi}$ satisfies
\begin{align}\label{case 1: weak soln system}
\left\{\begin{array}{ll}
(-\Delta)^s \tilde{\phi}-p U^{p-1} \tilde{\phi}=0, & \text { in } \mathbb{R}^{n} \backslash\left\{\bar{z}_{j} \mid j \in\left(T_{1}^{+} \cup C_{1}^{+}\right) \cap I_{1,1}\right\} \\ |\tilde{\phi}|\left(y_{1}\right) \lesssim\left|y_{1}-\bar{z}_{j}\right|^{4s-n}+1, & j \in\left(T_{1}^{+} \cup C_{1}^{+}\right) \cap I_{1,1}, \\ \int_{\R^n} U^{p-1} Z^{a} \tilde{\phi} d y_{1}=0, & 1 \leq a \leq n+1\end{array}\right.
\end{align}
The orthogonality condition implies $\tilde{\phi}$ is orthogonal to the kernel of the linearized operator $(-\Delta)^s - pU^{p-1}$ and thus $\tilde{\phi} =0.$ On the other hand since $|Y_k|= |\lambda_1(x_k-z_1)|\leq L_1$ upto a subsequence we have that $Y_k\to Y_{\infty}$ and thus $\tilde{\phi}(Y_{\infty}) = 1$ which is clearly a contradiction.

Case 2. Suppose $\{x_k\}_{k \in \N}\subset \Omega \setminus \bigcup_{i\in I}\Omega_i.$ Then there exists an index $i\in I$ such that there exists a subsequence $\{x_k\}_{k \in \N}$ satisfying
\begin{align*}
    \{x_k\}_{k \in \N}\subset A_i = \bigcup_{j \in T_{i}^{+} \cup C_{i}^{+}}\left\{\left|y_{i}\right| \leq L_i,\left|y_{i}-\tilde{z}_{i j}\right| \leq \epsilon,\left|y_{j}\right| \geq L_j\right\}.
\end{align*}
We assume w.l.o.g. that $i=1$. For $\mu\in (0,1/2)$ define the function
\begin{align}\label{min approx}
F(a, b)=\frac{a+b}{2}-\sqrt{\left(\frac{a-b}{2}\right)^{2}+\mu a b},
\end{align}
which approximates the function $\min\{a,b\}.$ Using $F$ we define new weight functions $\tilde{W}$ and $\tilde{V}$ as follows
\begin{align}\label{tilde w and v defn}
\tilde{W}(x)=\sum_{j \in J_{1}} \lambda_{j}^{\frac{n-2s}{2}} F\left(\frac{R^{2s-n}}{\left\langle y_{j}\right\rangle^{2s}}, \frac{R^{-4s}}{\left\langle y_{j}\right\rangle^{n-4s}}\right)+w_{1,1}(x),\\\tilde{V}(x)=\sum_{j \in J_{1}} \lambda_{j}^{\frac{n+2s}{2}} F\left(\frac{R^{2s-n}}{\left\langle y_{j}\right\rangle^{4s}}, \frac{R^{-4s}}{\left\langle y_{j}\right\rangle^{n-2s}}\right)+v_{1,1}(x).
\end{align}
Then 
\begin{align*}
\tilde{W}(x) &\approx \sum_{j \in J_{1}}\left(w_{j, 1}(x)+w_{j, 2}(x)\right)+w_{1,1}(x) \\ \tilde{V}(x) &\approx \sum_{j \in J_{1}}\left(v_{j, 1}(x)+v_{j, 2}(x)\right)+v_{1,1}(x),
\end{align*}
where $J_1 = T_{1}^{+} \cup C_{1}^{+}$. We need the following three propositions to conclude estimate \eqref{phi leq h} for Case 2.
\begin{proposition}\label{laplace prop}
The functions $\tilde{W}$ and $\tilde{V}$ satisfy
\begin{align}\label{laplace estimate}
    (-\Delta)^s \tilde{W}\geq \alpha_{n,s} \tilde{V},
\end{align}
where $\alpha_{n,s}>0$ is a constant depending on $n$ and $s.$
\end{proposition}

\begin{proposition}\label{claim 2}
For two bubbles $U_{i}$ and $U_{j}$, suppose $k$ large enough, in the region $\left\{\left|y_{i}\right| \geq L_i,\left|y_{j}\right| \geq L_j\right\}$ it holds that
\begin{align}\label{claim 2: estimate}
\sum_{m, l \in\{i,j\}} U_{m}^{p-1}\left(w_{l, 1}+w_{l, 2}\right) \leq \epsilon_{1} \sum_{l \in\{i,j\}}\left(v_{l, 1}+v_{l, 2}\right)
\end{align}
with $\epsilon_1>0$ depending on $L, n, s$ and $\epsilon.$
\end{proposition}

\begin{proposition}\label{claim 3}
For bubble $U_{i}$, let $J_{i}=T_{i}^{+} \cup C_{i}^{+} .$ In the region $A_{i}$, we have
\begin{align}\label{claim 3: estimate}
U_{i}^{p-1} \sum_{j \in J_{i}}\left(w_{j, 1}+w_{j, 2}\right) & \leq \epsilon_{1} \sum_{j \in J_{i}}\left(v_{j, 1}+v_{j, 2}\right) \\
\sum_{j \in J_{i}} U_{j}^{p-1} w_{i, 1} & \leq \epsilon_{1} \sum_{j \in J_{i}}\left(v_{j, 1}+v_{j, 2}\right)+\epsilon_{1} v_{i, 1}
\end{align}
when $k$ is large enough.
\end{proposition}
From Proposition \ref{laplace prop} we have $(-\Delta)^s \tilde{W}\geq \alpha_{n,s} \tilde{V}$. Furthermore in $A_1$ it is clear that $U_1 \gg \sum_{j\in T_1^{-}\cup C_1^{-}} U_j$ and therefore
\begin{align*}
\sigma_{k}^{p-1}=\left(\sum_{j \in J_{1}} U_{j}+U_{1}+\sum_{j \in T_{1}^{-} \cup C_{1}^{-}} U_{j}\right)^{p-1} \leq \sum_{j \in J_{1}} U_{j}^{p-1}+C U_{1}^{p-1}
\end{align*}
where $C=C(n,s)>0$ is constant depending on $n$ and $s.$ This estimate combined with \eqref{claim 2: estimate} and \eqref{claim 3: estimate} implies
\begin{equation}\label{non-linear term estimate}
\begin{aligned}
\sigma_{k}^{p-1} \tilde{W}\leq& \sum_{i, j \in J_{1}} U_{i}^{p-1}\left(w_{j, 1}+w_{j, 2}\right)+C U_{1}^{p-1} \sum_{j \in J_{1}}\left(w_{j, 1}+w_{j, 2}\right) \\ &+w_{1,1} \sum_{j \in J_{1}} U_{j}^{p-1}+C U_{1}^{p-1} w_{1,1} \\
&\leq \left(4\nu^2 \epsilon_{1}+C\right) \tilde{V},
\end{aligned}
\end{equation}
in the region $A_1.$ Thus choosing large $L$ and small $\epsilon$, we get
\begin{align*}
(-\Delta)^s \tilde{W}-p \sigma_{k}^{p-1} \tilde{W} \geq \tilde{V},\quad \text{ in } A_{1}.
\end{align*}
From Case 1, for large $k$, we have
\begin{align}\label{Proposition 1 estimate}
    |\phi_k|(x)\leq C \normds{h_k}W(x),\quad \forall x \in \bigcup_{i=1}^\nu\Omega_i
\end{align}
where $C=C(n,\nu,s)>0$ is a large constant. This is because if \eqref{Proposition 1 estimate} were not true then for any $\gamma\in \N$ there exists $k=k(\gamma)$ and $x_k\in A$ such that
\begin{align*}
    |\phi_k|(x_k)W^{-1}(x_k)>\gamma \normds{h_k}.
\end{align*}
W.l.o.g. assume that $|\phi_k|(x_k)W^{-1}(x_k)=1$, then $\normds{h_k}\leq \frac{1}{\gamma}\to 0$ as $\gamma\to +\infty$ and $k=k(\gamma)\to +\infty.$ However this cannot happen as shown in Case 1. Next, for $i\in T_{1}^{-}$ with $|y_i|=\frac{\lambda_i}{\lambda_1}|y_1- \tilde{z}_i|\leq L_i,$ for large $k$, we have
\begin{equation}\label{w estimate}
\begin{aligned}
\frac{w_{1,1}}{w_{i, 1}}=\left(\frac{\lambda_{1}}{\lambda_{i}}\right)^{\frac{n-2s}{2}} \frac{\left\langle y_{i}\right\rangle^{2s}}{\left\langle y_{1}\right\rangle^{2s}} \geq L^{-2s}\left(\frac{\lambda_{1}}{\lambda_{i}}\right)^{\frac{n-2s}{2}} \gg 1, \\
\frac{v_{1,1}}{v_{i, 1}}=\left(\frac{\lambda_{1}}{\lambda_{i}}\right)^{\frac{n+2s}{2}} \frac{\left\langle y_{i}\right\rangle^{4s}}{\left\langle y_{1}\right\rangle^{4s}} \geq L^{-4s}\left(\frac{\lambda_{1}}{\lambda_{i}}\right)^{\frac{n+2s}{2}} \gg 1.
\end{aligned}
\end{equation}
Similarly if $i\in C_1^{-}$ then $|y_i| =\frac{\lambda_2}{\lambda_1}|y_1- \tilde{z}_i| \geq \frac{\lambda_i}{2\lambda_1}|\tilde{z}_i| = \frac{\lambda_i|z_1-z_i|}{2}$, which in turn implies
\begin{equation}\label{v estimate}
\begin{aligned}
\frac{w_{1,1}}{w_{i, 2}}&=R^{6s-n}\left(\frac{\lambda_{1}}{\lambda_{i}}\right)^{\frac{n-2s}{2}} \frac{\left\langle y_{i}\right\rangle^{n-4s}}{\left\langle y_{1}\right\rangle^{2s}} \geq 2^{4s-n} L^{-2s} R^{2s}\left(\frac{\lambda_{1}}{\lambda_{i}}\right) \gg 1, \\
\frac{v_{1,1}}{v_{i, 2}}&=R^{6s-n}\left(\frac{\lambda_{1}}{\lambda_{i}}\right)^{\frac{n+2s}{2}} \frac{\left\langle y_{i}\right\rangle^{n-2s}}{\left\langle y_{1}\right\rangle^{4s}} \geq 2^{2s-n} L^{-4s} R^{4s}\left(\frac{\lambda_{1}}{\lambda_{i}}\right)^{2s} \gg 1.
\end{aligned}
\end{equation}
Combining estimates \eqref{w estimate} and \eqref{v estimate} we get the following estimate in the region $A_1$
\begin{equation}\label{W and V estimate in A_1}
\begin{aligned}
W &\approx \tilde{W} + \sum_{j\in T_1^{-}\cup C_1^{-}} (w_{j,1} + w_{j,2}) \approx \tilde{W},\\
V &\approx \tilde{V} + \sum_{j\in T_1^{-}\cup C_1^{-}} (v_{j,1} + v_{j,2}) \approx \tilde{V}.
\end{aligned}
\end{equation}
For large $k$, $\partial A_1 \subset \cup_{i\in I}\partial \Omega_i$ and therefore from \eqref{Proposition 1 estimate} and \eqref{W and V estimate in A_1} we conclude that
\begin{align*}
|\phi_k|(x) \leq C \normds{h_k} \tilde{W}(x),\quad \forall x \in \partial A_1
\end{align*}
which along with \eqref{non-linear term estimate} implies that $\pm C\normds{h_k}\tilde{W}$ is an upper/lower barrier for $\phi_k$ and therefore
\begin{align*}
|\phi_k|(x_k) \tilde{W}^{-1}(x_k) \lesssim \normds{h_k} \to 0. 
\end{align*}
This contradicts the fact that $|\phi_k|(x_k) =W(x_k).$ 

Case 3. Finally consider the case $\{x_k\}_{k \in \N}\subset \Omega^c = \bigcup_{i\in I}\{|y_i|> L_i\}.$ Working with similar approximations of $W$ and $V$ as defined in \eqref{tilde w and v defn} and arguing as in the proof of Proposition \ref{laplace prop} we get that $(-\Delta)^s \tilde{W}\geq \alpha_{n,s} \tilde{V}$. Next using Proposition \ref{claim 2}, we get
\begin{align*}
\sigma_k^{p-1} \tilde{W} (x) \leq \nu^2 \epsilon_1 \tilde{V} (x),\quad  \forall x\in \Omega^c.  
\end{align*}
Consequently in the region $\Omega^c$, we get
\begin{align}
    (-\Delta)^s \tilde{W}-p \sigma_{k}^{p-1} \tilde{W} \geq \tilde{V}.
\end{align}
From the previous two cases we know that for large $k$, we have
\begin{align*}
    |\phi_k|(x) \leq C \normds{h_k} \tilde{W}(x)
\end{align*}
in the region $\Omega$ and thus the above estimate also holds on the boundary $\partial \Omega = \partial \Omega^c.$ Thus $\pm C \normds{h_k} \tilde{W}(x)$ is an upper/lower barrier for the function $\phi_k$. This implies that
\begin{align*}
    |\phi_k|(x) \leq C \normds{h_k} \tilde{W}(x), \quad \forall x\in \Omega,
\end{align*}
which in turn implies 
\begin{align*}
    |\phi_k|(x_k) \tilde{W}^{-1}(x_k) \leq C \normds{h_k} \to 0.
\end{align*}
This contradicts the fact $|\phi_k|(x_k) =W(x_k).$ To complete the proof we prove Proposition \ref{laplace prop} since the proof of the other propositions can be found in \cite{deng2021sharp} by modifying the exponents of the weights by the parameter $s.$

\textbf{Proof of Proposition \ref{laplace prop}}. Using the concavity of $F$ and the integral representation of the fractional laplacian we deduce that
\begin{align*}
    (-\Delta)^s\Tilde{W}(x)\geq \sum_{j\in J_1}\lambda_j^{\frac{n-2s}{2}}\left(\pa{F}(a_j, b_j) (-\Delta)^s a_j + \pb{F}(a_j, b_j) (-\Delta)^s b_j\right),
\end{align*}
where $a_j = R^{2s-n}\jp{y_j}^{-2s}$, $b_j=R^{-4s}\jp{y_j}^{4s-n}$ and $y_j=\lambda_j(x-z_j).$ Thus to obtain a lower bound for $(-\Delta)^s\tilde{W}$ we first show that
\begin{align*}
    (-\Delta)^s \langle y_j\rangle^{-2s}  \geq c_{n,s} \langle y_j\rangle^{-4s},\quad (-\Delta)^s \langle y_j\rangle^{n-4s}  \geq c_{n,s} \langle y_j\rangle^{n-2s}.
\end{align*}
We prove the first estimate since the argument for the second inequality is the same. Furthermore, by scaling we can consider the case when $\lambda_j=1$ and $z_j=0$. Thus we need to show that
\begin{align*}
    (-\Delta)^s[(1+|x|^2)^{-s}]\geq \alpha_{n,s} [1+|x|^2]^{-2s}.
\end{align*}
Using the hypergeometric function as in Table 1 in \cite{Kwasnicki}, we get
\begin{align*}
    (-\Delta)^s[(1+|x|^2)^{-s}] = c_{n,s} {}_{2}F_{1}\left(n/2+s,2s,n/2;-|x|^2\right)
\end{align*}
for constant $c_{n,s}>0$ depending on $n$ and $s.$ By the integral representation as in (15.6.1) in \cite{dlmf}, we have
\begin{align*}
 {}_{2}F_{1}\left(n/2+s,2s,n/2;-|x|^2\right) = \frac{1}{\Gamma(2s)\Gamma(n/2-2s)}\int_{0}^{1}\frac{t^{2s-1}(1-t)^{n/2-2s-1}}{(1+|x|^2t)^{n/2+s}}dt\geq 0 
\end{align*}
as $n>4s.$ Next observe that the left-hand side has no roots. This is because if we use transformation (15.8.1) in \cite{dlmf} we get
\begin{align*}
{}_{2}F_{1}\left(n/2+s,2s,n/2;-|x|^2\right) = (1+|x|^2)^{-2s}{}_{2}F_{1}\left(-s,2s,n/2;\frac{|x|^2}{1+|x|^2}\right)
\end{align*}
and therefore we can compute the number of zeros using (15.3.1) in \cite{dlmf} as suggested in \cite{advice} to get
\begin{align*}
    N(n,s) = \lfloor s \rfloor + \frac{1}{2}(1+S) = \frac{1}{2}(1+S)
\end{align*}
where $S = \operatorname{Sign}(\Gamma(-s)\Gamma(2s)\Gamma(n/2+s)\Gamma(n/2-2s))=-1$. Then since $n>4s$, we get $N(n,s) = 0.$ Thus the function ${}_{2}F_{1}\left(n/2+s,2s,n/2;-|x|^2\right)$ must be strictly positive and in particular there exists a constant $\alpha_{n,s}>0$ such that
\begin{align*}
    (-\Delta)^s[(1+|x|^2)^{-s}] > \alpha_{n,s} \geq \alpha_{n,s}[1+|x|^2]^{-2s}
\end{align*} 
which gives us the desired estimate. Finally using the homogeneity of the function $F$ we can conclude the proof of Proposition \ref{laplace prop}.
\end{proof}

The next result gives an estimate for the coefficients $c^i_a$ of the system \eqref{first approx system}. 

\begin{lemma}\label{lemma coefficient bound}
Let $\sigma$ be a sum of $\delta-$interacting bubbles and let $\phi, h$ and $c^j_b$ for $j=1,\cdots,\nu$ and $b=1,2,\cdots, n+1$ satisfy the system \eqref{first approx system}. Then
\begin{align}\label{coefficient estimate}
    |c^j_b|\lesssim Q \normds{h} + Q^p\normss{\phi},
\end{align}
where $Q$ is the interaction term as defined in \eqref{defn: Q interaction term}.
\end{lemma} 
\begin{proof}
For simplicity we set $j=1$ and $\nu=2.$ Multiplying $Z^b_1$ to \eqref{first approx system} and integrating we get
\begin{align*}
    \int_{\R^n} p\sigma^{p-1}\phi Z^b_1 = \int_{\R^n} hZ^b_1 + \sum_{a=1}^{n+1}c^{1}_a \int_{\R^n} U_1^{p-1}Z_1^a Z^b_1 + c^{2}_a\int_{\R^n} U_2^{p-1}Z_2^a Z^b_1.
\end{align*}
Using Lemma \ref{Lemma A.5} we get
\begin{align*}
    \sum_{a=1}^{n+1} \int_{\R^n} c^{1}_a U_1^{p-1}Z_1^a Z^b_1 + \int_{\R^n} c^{2}_a U_2^{p-1}Z_2^a Z^b_1 = c^1_b \gamma^b + \sum_{a=1}^{n+1} c^{2}_a O(q_{12}).
\end{align*}
Using Lemma \ref{lemma B.4} we can also estimate the first term on the RHS as follows
\begin{align*}
    \left| \int_{\R^n} h Z_1^b\right|\lesssim \normds{h} \int_{\R^n} V U_1 \lesssim Q \normds{h}. 
\end{align*}
Moving to the LHS, since $\phi$ satisfies
\begin{align*}
    \int_{\R^n} U_1^{p-1} \phi Z_1^b = 0 
\end{align*}
for $b=1,2,\cdots, n+1$ and $|Z_1^b|\lesssim U_1$ we have
\begin{align*}
\left|\int_{\R^n} p \sigma^{p-1} \phi Z_1^{b} \right| &=\left|\int_{\R^n} p (\sigma^{p-1} - U_1^{p-1}) \phi Z_1^{b} \right| \lesssim \normss{\phi} \left|\int_{\R^n} (\sigma^{p-1} - U_1^{p-1}) U_1 W \right| \\
&\lesssim \normss{\phi} \left|\int_{\R^n} (\sigma^{p} - U_1^{p}-U_2^p) W \right| \lesssim \normss{\phi} \int_{\R^n} V W \lesssim \normss{\phi} Q^p,
\end{align*}
where we used Lemma \ref{lemma h bound in double star} to control the interaction term $\sigma^p-\sum_{i=1}^\nu U_i^p$ and Lemma \ref{Lemma B.2} to control the integral term $\int_{\R^n} V W.$ Thus we get that $\{c^1_b\}_{b=1}^{n+1}$ satisfies the following system
\begin{align*}
    c^1_b \gamma^b + \sum_{a=1}^{n+1}c_a^2 O(q_{12}) = \int_{\R^n} p\sigma^{p-1} \phi Z_1^b - \int_{\R^n} h Z_1^b
 \end{align*}
 and since $q_{12} < Q < \delta$ the above system is solvable such that the estimate \eqref{coefficient estimate} also holds. To prove the estimate for $\nu>2$ bubbles, one just needs to observe that for each $j\in I$
 \begin{align*}
\left(\sigma^{p-1}-U_{j}^{p-1}\right) U_{j} \leq \sum_{i=1}^{\nu}\left(\sigma^{p-1}-U_{i}^{p-1}\right) U_{i}=\sigma^{p}-\sum_{i} U_{i}^{p}
\end{align*}
and thus repeating the same argument as above one ends up with the following system
\begin{align*}
c_{b}^{1} \gamma^{b}+\sum_{i \neq 1} \sum_{a=1}^{n+1} c_{a}^{i} O\left(q_{i 1}\right)=\int p \sigma^{p-1} \phi Z_{1}^{b}+\int h Z_{1}^{b}
\end{align*}
which is a solvable system for $\delta$ small enough.
\end{proof}
Next, we prove that the system \eqref{first approx system} has a unique solution under some smallness conditions. 

\begin{lemma}\label{lemma: soln existence of approx system}
There exists a constant $\delta_0 > 0$ and $C>0$ independent of $\delta$ such that for $\delta_0 \leq \delta$ and any $h$ such that $\normds{h}<\infty$ the system \eqref{first approx system} has a unique solution $\phi \equiv L_{\delta}(h)$ such that
\begin{align}\label{lemma existence estimates}
    \normss{L_{\delta}(h)}\leq C \normds{h}, \quad |c^i_a| \leq \delta \normds{h}. 
\end{align}
\end{lemma}
\begin{proof}
We imitate the proof of Proposition 4.1 in \cite{delPino2003TwobubbleSI}. For this consider the space of functions 
\begin{align*}
    H = \{\phi \in \dot{H}^{s}(\R^n): \innerhs{\phi}{Z_i^a}=-p\int_{\R^n} U_i^{p-1} \phi Z_i^a = 0, i=1,\cdots,\nu, a =1,\cdots, n+1\},
\end{align*}
endowed with the natural inner product
\begin{align*}
    \inner{\phi}{\psi}_{H} =  \inner{\phi}{\psi}_{\dot{H}^s}= \frac{C_{n,s}}{2}\int_{\R^n} \int_{\R^n} \frac{(\phi(x)-\phi(y)) (\psi(x)-\psi(y)) }{|x-y|^{n+2s}} dx dy.
\end{align*}
In the weak form solving the system \eqref{first approx system} is equivalent to finding a function $\phi\in H$ that for all $\psi \in H$ satisfies
\begin{align*}
\inner{\phi}{\psi}_{H}  = \inner{-p\sigma^{p-1}\phi-h}{\psi}_{L^2} 
\end{align*}
which in operator form can be written as 
\begin{align*}
    \phi  = T_{\delta}(\phi) + \tilde{h}
\end{align*}
where $\tilde{h}$ depends linearly on $h$ and $T_\delta$ is a compact operator on $H.$ Then Fredholm's alternative implies that there exists a unique solution $\phi\in H$ to the above equation provided the only solution to 
\begin{align*}
    \phi = T_{\delta}(\phi)
\end{align*}
is $\phi \equiv 0$ in $H.$ In other words we want to show that the following equation has a trivial solution in $H$
\begin{align*}
    (-\Delta)^s \phi - p\sigma^{p-1}\phi +\sum_{i, a} c^i_a U_i^{p-1}Z_i^a=0. 
\end{align*}
We proceed by contradiction. Suppose there exists a non-trivial solution $\phi\equiv \phi_{\delta}$. W.l.o.g. assume that $\normss{\phi_{\delta}} =1.$ However from Lemma \ref{first approx: baby existence lemma} and Lemma \ref{lemma coefficient bound} we get that $\normss{\phi_{\delta}}\to 0$ as $Q \to 0$ since
\begin{align*}
    \normss{\phi_{\delta}} \leq C \sum_{i,a}|c_a^i|\leq  C' Q^p\normss{\phi_{\delta}}\lesssim Q^p,
\end{align*}
which is a contradiction. Thus for each $h$, the system \eqref{first approx system} admits a unique solution in $H.$ Furthermore the estimates \eqref{lemma existence estimates} also follow from Lemma \ref{first approx: baby existence lemma} and Lemma \ref{lemma coefficient bound}. 
\end{proof}
With this Lemma in hand, we can now prove the main result of this section.
\begin{proposition}\label{prop: rho0 existence and pointwise estimate}
Suppose $\delta$ is small enough. There exists $\rho_0$ and a family of scalar $\{c^i_a\}$ which solves \eqref{eqn: rho0 PDE}
with
\begin{align}\label{estimate: rho0}
    |\rho_0(x)|\leq C W(x).
\end{align}
\end{proposition}
\begin{proof}
Set
\begin{align*}
    N_{1}(\phi) = (\sigma+\phi)^p-\sigma^p - p\sigma^{p-1}\phi, \quad N_2 = \sigma^p - \sum_{i=1}^\nu U_i^p
\end{align*}
and denote $L_{\delta}(h)$ to be the solution to the system \eqref{first approx system}.
Then
\begin{align*}
    (-\Delta)^s \phi - p\sigma^{p-1}\phi= N_1(\phi) + N_2- \sum_{i,a} c^{i}_a U^{p-1}_i Z_i^a=0.
\end{align*}
Thus solving \eqref{eqn: rho0 PDE} is the same as solving
\begin{align*}
    \phi = A(\phi) = L_{\delta}(N_1(\phi)) +L_{\delta}(N_2)
\end{align*}
where $L_{\delta}$ is defined in Lemma \ref{lemma: soln existence of approx system}. We show that $A$ is a contraction map in a suitable normed space and show the existence of a solution $\rho_0$ using the fixed point theorem. First using Lemma \ref{lemma h bound in double star} we have that
\begin{align*}
    \normds{N_2} = \normds{h} \leq C_2
\end{align*}
for some large constant $C_2 = C_2(n,s)>0.$  To control $N_1$ observe that
\begin{align*}
    |N_1(\phi)| \leq C|\phi|^p \leq C \normss{\phi} W^p
\end{align*}
which implies that 
\begin{align*}
    \normds{N_1(\phi)} \leq C \normss{\phi} \sup_{x\in \R^n} W^p(x) V^{-1}(x) \leq C_1 R^{-4(p-1)} \normss{\phi}
\end{align*}
for some large constant $C_1= C_1(n,s)>0$ that also satisfies the estimate
\begin{align*}
    \normss{L_{\delta}(h)}\leq C_1 \normds{h}.
\end{align*}
Define the space
\begin{align*}
    E = \{u\in C^{1}(\R^n) \cap \dot{H}^{s}(\R^n): \normss{u}\leq C_1C_2 + 1\}.
\end{align*}
We show that the operator $A$ is a contraction on the space $(E, \normss{\cdot}).$ First we show that $A(E)\subset E.$ For this take $\phi \in E.$ Then 
\begin{align*}
    \normss{A(\phi)} &= \normss{L_{\delta}(N_1(\phi)) + L_{\delta}(N_2)} \leq C_1 \normds{N_1(\phi)} + C_1\normds{N_2}\\
    &\leq C_1^2R^{-4(p-1)}(C_1C_2 + 1) + C_1 C_2 \leq C_1 C_2 + 1
\end{align*}
for $R\gg 1$ when $\delta$ is small. Next, we show that $A$ is a contraction map. Observe that for $\phi_1, \phi_2\in E$ we have
\begin{align*}
\normss{A(\phi_1) - A(\phi_2)} \leq \normds{N_1(\phi_1) - N_1(\phi_2)}.
\end{align*}
Since when $n\geq 6s$, $|N'(t)|\leq C |t|^{p-1}$ which implies
\begin{align*}
|N_1(\phi_1) - N_1(\phi_2)| V^{-1} &\leq C(|\phi_1|^{p-1}+|\phi_2|^{p-1})|\phi_1 -\phi_2|\\
&\leq C (\normss{\phi_1}^{p-1} + \normss{\phi_2}^{p-1}) \normss{\phi_1 -\phi_2} W^{p}V^{-1} \\
&\leq \frac{1}{2} \normss{\phi_1 -\phi_2} ,
\end{align*}
where we choose small enough $\delta$ such that $R\gg 1.$ Then we get
\begin{align*}
\normss{A(\phi_1) - A(\phi_2)} \leq \normds{N_1(\phi_1) - N_1(\phi_2)} \leq \frac{1}{2}\normss{\phi_1 -\phi_2}.
\end{align*}
This shows that the equation $\phi = A(\phi)$ has a unique solution. Finally from Lemma \ref{lemma: soln existence of approx system} we get
\begin{align*}
    \normss{\phi} &= \normss{A(\phi)} \leq \normss{L_{\delta}(N_1)(\phi)} + \normss{L_{\delta}(N_2)} \\
    &\leq \normds{N_1(\phi)} + \normds{N_2} \leq C_1C_2 + 1 \leq C
\end{align*}
for a large constant $C>0.$
\end{proof}

\subsection{Energy Estimates of the First Approximation $\rho_0$}
In this section, our goal is to establish $L^2$ estimate for $(-\Delta)^{s/2} \rho_0$ where $\rho_0$ is the solution of the system \eqref{eqn: rho0 PDE} as in Proposition \ref{prop: rho0 existence and pointwise estimate}. 

Using Lemma \ref{Lemma B.2}, Lemma \ref{Lemma B.3}, and Proposition \ref{prop: rho0 existence and pointwise estimate} we obtain the following estimate. 
\begin{proposition}\label{prop: rho0 grad L2 estimate}
Suppose $\delta$ is small enough. Then, for $n\geq 6s$
\begin{align}\label{rho0 L2 gradient estimate}
\|(-\Delta)^{s/2}\rho_0\|_{L^2}\lesssim \begin{cases}Q|\log Q|^{\frac12}, &\text{if }n=6s,\\ Q^{\frac{p}{2}},&\text{if }n> 6s.\end{cases}.
\end{align}
\end{proposition}
\begin{proof}
Testing the equation in \eqref{eqn: rho0 PDE} with $\rho_0$ we get 
\begin{align*}
    \norml{(-\Delta)^{s/2}\rho_0 }^2 \lesssim \int_{\R^n} \sigma^{p-1}\rho_0 + \int_{\R^n} |\rho_0|^{p+1} + \int_{\R^n} (\sigma -\sum_{i=1}^{\nu}U_i^p) \rho_0,
\end{align*}
where we used the inequality $|(\sigma + \rho_0)-\sigma^p-p\sigma^{p-1}\rho_0|\lesssim |\rho_0|^p.$ Using $\sigma^{p-1} \lesssim U_1^{p-1} + \cdots +U_\nu^{p-1}$, $|\rho_0(x)|\lesssim W(x)$ from Proposition \ref{prop: rho0 existence and pointwise estimate} and Lemma \ref{Lemma B.1} we get
\begin{align*}
    \int_{\R^n} \sigma^{p-1} \rho_0 &\lesssim \int_{\R^n} \sum_{i=1}^{\nu} U_i^{p-1} W^2 \\
    &\lesssim \sum_{i=1}^{\nu}\sum_{j=1}^{2}\int_{\R^n} U_i^{p-1}w^2_{j,1}+U_i^{p-1}w^2_{j,2} \lesssim R^{-n-2s} \approx Q^{p}.
\end{align*}
For the second term using the Sobolev inequality we have
\begin{align}\label{rho0 grad L2 estimate: sobolev term}
\int_{\R^n} |\rho_0|^{p+1} \lesssim \norml{(-\Delta)^{s/2} \rho_0}^{p+1}.
\end{align}
Finally for the interaction term recall that Lemma \ref{lemma h bound in double star} implies that $|h| \lesssim V(x).$ Using this estimate along with $|\rho_0| \lesssim W(x)$ we get
\begin{align*}
    \int_{\R^n} |\sigma- \sum_{i=1}^{\nu}U_i^p - U_2^p)\rho_0| &\lesssim \int_{\R^n} V W\\
    &\approx \sum_{i=1}^{\nu}\sum_{j=1}^{2} \int_{\R^n} v_{i,1}w_{j,1} + v_{i,1}w_{j,2} + v_{i,2}w_{j,1} + v_{i,2}w_{j,2}.
\end{align*}
Using Lemma \ref{Lemma B.2} we get that the RHS is bounded by (up to a constant) $R^{-n-2s} \approx Q^{p}$ when $n\geq 6s$. When $n=6s$ we can obtain a more precise estimate using Lemma \ref{Lemma B.3} since $p=2$
\begin{align*}
    \int_{\R^n} |((U_1+U_2)^2-U_1^2 -U_2^2) \rho_0| \lesssim \int_{\R^n} U_1 U_2 W \lesssim R^{-8}\log R \approx Q^{2}|\log Q|.
\end{align*}

This concludes the proof of Proposition \ref{prop: rho0 grad L2 estimate}.
\end{proof}

\subsection{Energy Estimate of the Second Approximation $\rho_1$}

In our attempt to estimate the energy of the error term $\rho$ as defined from the minimization process in \eqref{minimization process}, we define $\rho_1 = \rho-\rho_0.$ Then since $\rho$ satisfies \eqref{rho PDE} and $\rho_0$ satisfies \eqref{eqn: rho0 PDE} the second approximation $\rho_1$ satisfies
\begin{align}\label{rho1 PDE}
\left\{\begin{array}{l}
(-\Delta)^s \rho_{1}-\left[\left(\sigma+\rho_{0}+\rho_{1}\right)^{p}-\left(\sigma+\rho_{0}\right)^{p}\right]-\sum_{j=1}^{\nu} \sum_{a=1}^{n+1} c_{a}^{j} U_{j}^{p-1} Z_{j}^{a}-f=0 \\ \innerhs{\rho_1}{Z_j^a}=0, \quad j=1,\cdots,\nu;\quad  a=1, \cdots, n+1\end{array}\right.
\end{align}
Here recall that $f= (-\Delta)^s u -u|u|^{p-1}.$ We further decompose $\rho_1$ as
\begin{align}\label{rho2 defn}
    \rho_1=\sum_{i=1}^{\nu}\beta^i U_i+\sum_{i=1}^{\nu}\sum_{a=1}^{n+1}\beta_a^iZ_{i}^a+\rho_2,
\end{align}
where $\rho_2$ satisfies the following orthogonality conditions
\begin{align}\label{rho2 ortho cond}
    \innerhs{\rho_2}{U_i}=\innerhs{\rho_2}{Z_i^a} = 0,
\end{align}
for $i=1,\cdots,\nu$ and $a=1,\cdots, n+1.$ Thus in order to estimate the $L^2$ norm of $(-\Delta)^{s/2} \rho_1$ we estimate the $L^2$ norm of $(-\Delta)^{s/2} \rho_2.$
\begin{lemma}\label{lem: rho2 grad estimate}
Suppose $\delta$ is small enough. Then \begin{align}\label{rh02 L2 grad estimate}
     \|(-\Delta)^{s/2}\rho_2\|_{L^2}\lesssim \sum_{i=1}^{\nu}|\beta^i|+\sum_{i=1}^{\nu}\sum_{a=1}^{n+1}|\beta_a^i|+\|f\|_{H^{-s}}.
\end{align}
\end{lemma}
\begin{proof}
Testing \eqref{rho1 PDE} with $\rho_2$ and using \eqref{rho2 ortho cond} we get
\begin{align*}
\int_{\R^n} |(-\Delta)^{s/2} \rho_2|^2 \lesssim \int_{\R^n} |(\sigma+\rho_0+\rho_1)^p-(\sigma+\rho_0)^p|\rho_2 + \int_{\R^n} |f \rho_2|.
\end{align*}
The second term can be trivially estimated as follows
\begin{align}\label{rho02 estimate second term}
\int_{\R^n} |f \rho_2| \lesssim \normd{f}\norml{(-\Delta)^{s/2} \rho_2}.
\end{align}
To estimate the first term, we make use of the following inequality
\begin{align*}
    |(\sigma + \rho_0 + \rho_1)^p - (\sigma+\rho_0)^p -  p(\sigma + \rho_0)^{p-1}\rho_1| \leq |\rho_1|^{p}
\end{align*}
and therefore we get
\begin{align*}
\int_{\R^n} |(\sigma+\rho_0+\rho_1)^p-(\sigma+\rho_0)^p|\rho_2 \leq p\int_{\R^n} |\sigma+\rho_0|^{p-1}|\rho_1\rho_2| + \int_{\R^n} |\rho_1|^{p}|\rho_2|.       
\end{align*}
Using the decomposition of $\rho_1$ in \eqref{rho2 defn} we get
\begin{align*}
    p\int_{\R^n} |\sigma + \rho_0|^{p-1}|\rho_1\rho_2|\leq p\int_{\R^n} |\sigma + \rho_0|^{p-1}|\rho_2|^2 + \mathcal{B}\int_{\R^n} |\sigma + \rho_0|^{p-1} U_i |\rho_2| 
\end{align*}
where $\mathcal{B} = \sum_i|\beta^i| + \sum_{i,a} |\beta^{i}_a|.$ For the first term, we have
\begin{align*}
p\int_{\R^n} |\sigma + \rho_0|^{p-1}|\rho_2|^2 &\leq p\int_{\R^n} \sigma^{p-1}|\rho_2|^2 + p\int_{\R^n} |\rho_0|^{p-1}|\rho_0|^2 \\
&\leq (\tilde{c}  + C \norml{(-\Delta)^{s/2} \rho_0}^{p-1}))\norml{(-\Delta)^{s/2} \rho_2}^2 
\end{align*}
where we made use of the Sobolev inequality and the spectral inequality of the same form as \eqref{fractional-bubble-spectral-estimate} with constant $\tilde{c}<1.$
For the other terms, using Sobolev inequality we have
\begin{align*}
    \int_{\R^n} |\sigma+\rho_0|^{p-1} U_i |\rho_2| \lesssim \norm{(\sigma + \rho_0) U_i}_{L^{\frac{2n}{n+2s}}} \norml{(-\Delta)^{s/2} \rho_2} \lesssim \norml{(-\Delta)^{s/2} \rho_2} 
\end{align*}
and 
\begin{align*}
    \int_{\R^n} |\rho_1|^p |\rho_2| \lesssim \norm{\rho_1}^p_{L^{2^*}} \norm{\rho_2}_{L^{2^*}} \lesssim (\mathcal{B} + \norml{(-\Delta)^{s/2} \rho_2})^p \norml{(-\Delta)^{s/2} \rho_2}.
\end{align*}
Thus combining the above estimates we get
\begin{align}\label{rho02 estimate first term}
    p\int_{\R^n} |\sigma + \rho_0|^{p-1}|\rho_1\rho_2| &\lesssim (\tilde{c}+C\norml{(-\Delta)^{s/2} \rho_0}^{p-1})\norml{(-\Delta)^{s/2} \rho_2}^2 + \mathcal{B}\norml{(-\Delta)^{s/2} \rho_2} \nonumber \\
    &\quad + (\mathcal{B} + \norml{(-\Delta)^{s/2} \rho_2})^p \norml{(-\Delta)^{s/2} \rho_2}.
\end{align}
For small enough $\delta$ using Proposition \ref{prop: rho0 grad L2 estimate} we can make $\norml{(-\Delta)^{s/2} \rho_0}\ll 1.$ Furthermore for small enough $\delta$ we can also make $\mathcal{B}<1$ and $\norml{(-\Delta)^{s/2} \rho_2}<1$ and thus using \eqref{rho02 estimate first term} and \eqref{rho02 estimate second term} we get \eqref{rh02 L2 grad estimate}.
\end{proof}
Estimate \eqref{rh02 L2 grad estimate} for $\rho_2$ suggests that the next natural step is to control the absolute sum of coefficients $\mathcal{B}$ as defined in \eqref{rho2 defn}.
\begin{lemma}\label{lem:beta coeff bound}
If $\delta$ is small, then
\begin{align}\label{eqn:beta coeff estimate}
|\beta^i|+|\beta_a^i|\lesssim Q^2+\|f\|_{H^{-s}}.
\end{align}
\end{lemma}
\begin{proof}
We start by multiplying the bubble $U_k$ to equation \eqref{rho1 PDE} and integrating by parts. Thus we get
\begin{align*}
    &\int_{\R^n} (-\Delta)^{s/2} \rho_1 \cdot (-\Delta)^{s/2} U_k + \int_{\R^n}[(\sigma + \rho_0 + \rho_1)^p - (\sigma +\rho_0)^p]U_k \\
    &+ \sum_{i,a}c^{i}_a U_i^{p-1}U_k Z_{i}^a + \int_{\R^n} U_k f=0.
\end{align*}
Using the inequality 
\begin{align*}
    |(\sigma+\rho_0)^{p-1}- U_k^{p-1}|\lesssim \sum_{i\neq k} U_i^{p-1} + |\rho_0|^{p-1},
\end{align*}
we get the following estimate
\begin{align*}
\begin{split}
    &\left|\int_{\R^n}[(\sigma+\rho_0+\rho_1)^p-(\sigma+\rho_0)^p]U_k-p\int_{\R^n} U_k^{p}\rho_1\right|\\
    &\hspace{1cm}\lesssim
    \sum_{i\neq k}\int_{\R^n} U_i^{p-1}U_k|\rho_1|+ \int_{\R^n} |\rho_1|^p U_k+\int_{\R^n} |\rho_0|^{p-1}\rho_1U_k.
\end{split}
\end{align*}
Estimating the three terms separately we get
\begin{align*}
    \int_{\R^n} U_i^{p-1}|\rho_1|U_k&\leq \|(-\Delta)^{s/2} \rho_1\|_{L^{2}}\|U_i^{p-1}U_k\|_{L^{\frac{2n}{n+2s}}}\lesssim o(1)\left(\mathcal{B}+\|f\|_{H^{-s}}\right),\text{ when }i\neq k \\
    \int_{\R^n} |\rho_1|^p U_k&\lesssim \mathcal{B}^p+\int_{\R^n}\rho_2^pU_k\lesssim \mathcal{B}^p + \norml{(-\Delta)^{s/2} \rho_2}^p \lesssim \mathcal{B}^p+\|f\|_{H^{-s}}^p,\\
    \int_{\R^n}|\rho_0|^{p-1}\rho_1U_k&\lesssim \|(-\Delta)^{s/2}\rho_0\|_{L^2}^{p-1}\|(-\Delta)^{s/2}\rho_1\|_{L^2}\lesssim o(1)\left(\mathcal{B}+\|f\|_{H^{-s}}\right),
\end{align*}
where $o(1)$ denotes a quantity that tends to $0$ when $\delta\to 0.$ Next, using integral estimate similar to Proposition B.2 in Appendix B of \cite{FigGla20} and $|Z_j^a|\lesssim U_j$ we get that
\begin{align*}
    \int_{\R^n} U_j^{p-1}Z_j^a U_k \lesssim \int_{\R^n} U_j^{p}U_k \lesssim Q
\end{align*}
when $j\neq k$ otherwise the above integral is equal to $0.$ Furthermore, the coefficients $c^i_a$ can be estimated using Lemma \ref{lemma h bound in double star} and 
Lemma \ref{lemma coefficient bound} to get
\begin{align*}
    |c^i_a| \lesssim Q\normds{h} + Q^p\normss{\rho_0} \lesssim Q. 
\end{align*}
Thus we get
\begin{align}\label{lhs bound}
    \int_{\R^n} (-\Delta)^{s/2} \rho_1(-\Delta)^{s/2} U_k+p\int_{\R^n} U_k^p\rho_1\lesssim& o(1)\mathcal{B}+Q^2+\normd{f}.
\end{align}
Writing  $\int_{\R^n} (-\Delta)^{s/2} \rho_1(-\Delta)^{s/2} U_k+p\int_{\R^n} U_k^p\rho_1 = (p-1)\int_{\R^n} U_k^p\rho_1$ and using the decomposition of $\rho_1$ in \eqref{rho2 defn} along with the orthogonality condition for $\rho_2$ in \eqref{rho2 ortho cond} we get
\begin{align}\label{lhs decomp}
    \int_{\R^n} (-\Delta)^{s/2} \rho_1(-\Delta)^{s/2} U_k+p\int_{\R^n} U_k^p\rho_1 = (p-1)\left(\sum_i \beta^i\int_{\R^n} U_i^p U_k+\sum_{i,a} \beta_a^i\int_{\R^n} U_k^{p}Z_i^a \right) \nonumber \\
    = (p-1)\beta^k S^{p+1} + \sum_{i\neq k}\beta^i O(q_{ik}) + \sum_{i\neq k,a}\beta^{i}_a O(q_{ik})
\end{align}
where we made use of the integral estimate similar to Proposition B.2 in Appendix B of \cite{FigGla20} in the last step. On the other hand the orthogonality condition $\innerhs{\rho_1}{Z_k^b} =0$, \eqref{rho2 defn}, \eqref{rho2 ortho cond} and Lemma \ref{Lemma A.5} implies
\begin{align}\label{zero ortho cond}
0 &=\sum_{i}\beta^i\int_{\R^n} (-\Delta)^{s/2} U_i\cdot(-\Delta)^{s/2} Z^b_k+\sum_{i,a}\beta_a^i\int_{\R^n} (-\Delta)^{s/2} Z_i^a\cdot (-\Delta)^{s/2} Z^b_k \nonumber\\
&= \sum_{i\neq k}\beta^i O(q_{ik}) + \beta^k_b +  \sum_{i\neq k, a\neq b} \beta^i_a O(q_{ik}).
\end{align}
Thus combining \eqref{zero ortho cond} and \eqref{lhs decomp} we get 
\begin{align*}
  &(p-1)\beta^k S^{p+1} + \sum_{i\neq k}\beta^i O(q_{ik}) + \sum_{i\neq k,a}\beta^{i}_a O(q_{ik}) \\
  &=   (p-1)\beta^k S^{p+1}  + \sum_{i\neq k}\beta^{i} O(q_{ik}) + \sum_{i\neq k,a\neq b}\beta^{i}_a O(q_{ik}) + \sum_{i\neq k}\beta^{i}_b O(q_{ik}) \\
  &= (p-1)\beta^k S^{p+1} -\beta^k_b + \sum_{i\neq k} \beta^i_b O(q_{ik}) 
\end{align*}
which along with \eqref{lhs bound} gives us the desired bound.
\end{proof}

\begin{lemma}\label{lemma: rho1 final grad estimate}
Let $\delta$ be small enough. Then,
\begin{align}\label{rho1 final grad estimate}
\|(-\Delta)^{s/2}\rho_1\|_{L^2}\lesssim Q^2+\|f\|_{H^{-s}}.    
\end{align}
\end{lemma}
\begin{proof}
From \eqref{rho2 defn}, \eqref{rh02 L2 grad estimate} and \eqref{eqn:beta coeff estimate} we get that
\begin{align*}
    \norml{(-\Delta)^{s/2} \rho_1} \lesssim \mathcal{B} + \norml{(-\Delta)^{s/2}\rho_2} \lesssim \mathcal{B} + \norm{f}_{H^{-s}} \lesssim Q^2 +\norm{f}_{H^{-s}}.
\end{align*}
\end{proof}
\subsection{Conclusion}Assuming Step (\rm{i}), (\rm{ii}), and (\rm{iii}) we conclude the proof of Theorem \ref{fractional-theorem-main} when $n\geq 6s$.
\begin{proof}
The proof proceeds in four steps.

Step 1. Recall the equation satisfied by the error $\rho$
\begin{align*}
    (-\Delta)^s\rho-p\sigma^{p-1}\rho-I_1-I_2-f=0
\end{align*}
where,
\begin{align*}
     f &= (-\Delta)^s u -u|u|^{p-1},\quad  I_1=\sigma^p- \sum_{i=1}^{\nu}U_i^p, \nonumber \\ I_2&=(\sigma+\rho)|\sigma+\rho|^{p-1}-\sigma^p-p\sigma^{p-1}\rho.
\end{align*}
Multiplying $Z_k^{n+1}$, integrating by parts and using the orthogonality condition for $\rho$ we get
\begin{align*}
    \left|\int_{\R^n} I_1 Z^{n+1}_k\right| \leq \int_{\R^n} p\sigma^{p-1}|\rho Z^{n+1}_k| + \int_{\R^n} |I_2 Z^{n+1}_k| + \int_{\R^n} |f Z^{n+1}_k|.
\end{align*}
Using $|I_2|\lesssim |\rho|^p$ for the second term and $|Z^{n+1}_k|\lesssim U_k$ for the third term we get
\begin{align}\label{I1 estimate}
    \left|\int_{\R^n} I_1 Z^{n+1}_k\right| \lesssim \int_{\R^n} \sigma^{p-1}|\rho Z^{n+1}_k| + \int_{\R^n} |\rho|^p |Z^{n+1}_k| + \normd{f}.
\end{align}
Next, we further estimate the first two terms in the above estimate.

Step 2. For $\delta$ small from Lemma 3.13 in \cite{deng2021sharp} we deduce that
\begin{align}\label{sigma_p and rho_p estimate}
    \left|\int_{\R^n} \sigma^{p-1} \rho Z_k^{n+1}\right|=o(Q)+\|f\|_{H^{-s}},\quad \int_{\R^n} |\rho|^{p} |Z_k^{n+1}|=o(Q)+\|f\|_{H^{-s}},
\end{align}
where $Q$ is defined as in \eqref{defn: Q interaction term} and $o(Q)$ denotes a quantity that satisfies $o(Q)/Q\to 0$ as $Q\to 0.$

Step 3. Thus from \eqref{I1 estimate}, \eqref{sigma_p and rho_p estimate} we get
\begin{align}\label{I1 final estimate}
    \int_{\R^n} I_1 Z_k^{n+1} \lesssim o(Q) + \normd{f}.
\end{align}
Thus using Lemma 2.1 from \cite{deng2021sharp} we get that
\begin{align}\label{I1 asymptotic expression}
    \int_{\R^n} I_1 Z_k^{n+1} = \int_{\R^n} I_1 \lambda_k \partial_{\lambda_k} U_k = \int_{\R^n} U_i^p \lambda_k \partial_{\lambda_k} U_k + o(Q) 
\end{align}
where $i\neq k.$ 

Step 4. Arguing as in Lemma 2.3 in \cite{deng2021sharp} we deduce that $Q\lesssim \normd{f}.$ Thus using the fact that, $\rho = \rho_0 + \rho_1$ and estimates \eqref{rho0 L2 gradient estimate} and \eqref{rho1 final grad estimate}  we get that
\begin{align*}
    \norml{(-\Delta)^{s/2} \rho} &\leq \norml{(-\Delta)^{s/2} \rho_0} + \norml{(-\Delta)^{s/2} \rho_1}\\ &\lesssim \begin{cases} \|f\|_{H^{-s}}\left|\log \|f\|_{H^{-s}}\right|^{\frac12},& n=6s,\\\|f\|_{H^{-s}}^{\frac{p}{2}},&n> 6s\end{cases}
\end{align*}
which concludes the proof of Theorem \ref{fractional-theorem-main}.
\end{proof}
\section*{Appendix}
Here we recall some integral estimates from Appendix A and B in \cite{deng2021sharp}. The only difference is that the exponents have been modified by the parameter $s.$ Thus one can follow the same proof as in Appendix A and B of \cite{deng2021sharp} except minor modifications.

\begin{lemma}\label{Lemma A.5}
For the $Z_{i}^{a}$ defined in \eqref{deng-z-derivative-notation}, there exist some constants $\gamma^{a}=\gamma^{a}(n,s)>0$ such that
$$
\int_{\R^n} U_{i}^{p-1} Z_{i}^{a} Z_{i}^{b}=\left\{\begin{array}{ll}
0 & \text { if } a \neq b \\
\gamma^{a} & \text { if } 1 \leq a=b \leq n+1
\end{array}\right.
$$
If $i \neq j$ and $1 \leq a, b \leq n+1$, we have
$$
\left|\int_{\R^n} U_{i}^{p-1} Z_{i}^{a} Z_{j}^{b}\right| \lesssim q_{i j},
$$
where for $i\neq j$, $q_{ij}=\left( \frac{\lambda_i}{\lambda_j} + \frac{\lambda_j}{\lambda_i} + \lambda_i\lambda_j|z_i-z_j|^2\right)^{-(n-2s)/2}.$
\end{lemma}

\begin{lemma}\label{Lemma B.2}
Suppose $n\geq 6s$ and $R\gg 1$, we have
\begin{align}
\int_{\R^n} v_{1,1} w_{1,1} &= \int_{|y_1|\leq R} \frac{\lambda_1^{\frac{n+2s}{2}}R^{2s-n}}{\langle y_1\rangle^{4s}}\frac{\lambda_1^{\frac{n-2s}{2}}R^{2s-n}}{\langle y_1\rangle^{2}}dx\approx\
\begin{cases}
  R^{-8}\log R,\quad &n=6s,\\
  R^{-n-2s},\quad &n> 6s,\\
\end{cases}\label{s-c-1}\\
\int_{\R^n} v_{1,2} w_{1,2} &= \int_{|y_1|\geq R} \frac{\lambda_1^{\frac{n+2s}{2}}R^{-4s}}{\langle y_1\rangle^{n-2s}}\frac{\lambda_1^{\frac{n-2s}{2}}R^{-4s}}{\langle y_1\rangle^{n-4s}}dx\approx R^{-n-2s},\quad n> 6s,
\label{s-c-2}\\
 \int_{\R^n} v_{1,1} w_{2,1} &= \int_{|y_1|\leq R,|y_2|\leq R} \frac{\lambda_1^{\frac{n+2s}{2}}R^{2s-n}}{\langle y_1\rangle^{4s}}\frac{\lambda_2^{\frac{n-2s}{2}}R^{2s-n}}{\langle y_2\rangle^{2s}}dx
 \lesssim\ \  R^{-n-2s},\quad n\geq 6s,\label{s-c-3}\\
\int_{\R^n} v_{1,1}w_{2,2} &= \int_{|y_1|\leq R,|y_2|\geq R} \frac{\lambda_1^{\frac{n+2s}{2}}R^{2s-n}}{\langle y_1\rangle^{4s}}\frac{\lambda_2^{\frac{n-2s}{2}}R^{-4s}}{\langle y_2\rangle^{n-4s}}dx\lesssim\
    \begin{cases}
            R^{-8}\log R,\quad &n=6s,\\
            R^{-n-2s},\quad &n> 6s,
    \end{cases}
    \label{s-c-4}\\
\int_{\R^n} v_{1,2} w_{2,1} &=     \int_{|y_1|\geq R,|y_2|\leq R} \frac{\lambda_1^{\frac{n+2s}{2}}R^{-4s}}{\langle y_1\rangle^{n-2s}}\frac{\lambda_2^{\frac{n-2s}{2}}R^{2s-n}}{\langle y_2\rangle^{2s}}dx\lesssim\
     \begin{cases}
            R^{-8}\log R,\quad &n=6s,\\
            R^{-n-2s},\quad &n> 6s,
        \end{cases}\label{s-c-5}\\
    \int_{\R^n} v_{1,2} w_{2,2} &=\int_{|y_1|\geq R,|y_2|\geq R} \frac{\lambda_1^{\frac{n+2s}{2}}R^{-4s}}{\langle y_1\rangle^{n-2s}}\frac{\lambda_2^{\frac{n-2s}{2}}R^{-4s}}{\langle y_2\rangle^{n-4s}}dx\lesssim\ \ R^{-n-2s},\quad n> 6s.
     \label{s-c-6}
\end{align}
\end{lemma}

\begin{lemma}\label{lemma B.4}
Suppose $n\geq 6s$ and $R \gg 1$, we have
\begin{align}
\int_{|y_1|\leq R} \frac{\lambda_1^{\frac{n+2s}{2}}R^{2s-n}}{\langle y_1\rangle^4}\frac{\lambda_1^{\frac{n-2s}{2}}}{\langle y_1\rangle^{n-2s}}dx\thickapprox&\ R^{2s-n},
\label{V-U-1}\\
\int_{|y_1|\geq R} \frac{\lambda_1^{\frac{n+2s}{2}}R^{-4s}}{\langle y_1\rangle^{n-2s}}\frac{\lambda_1^{\frac{n-2s}{2}}}{\langle y_1\rangle^{n-2s}}dx\thickapprox&\ R^{-n},
\label{v-U-2}\\
 \int_{|y_1|\leq R} \frac{\lambda_1^{\frac{n+2s}{2}}R^{2s-n}}{\langle y_1\rangle^{4s}}\frac{\lambda_2^{\frac{n-2s}{2}}}{\langle y_2\rangle^{n-2s}}dx
 \lesssim&\ R^{-n},\label{V-U-3}\\
    \int_{|y_1|\geq R} \frac{\lambda_1^{\frac{n+2s}{2}}R^{-4s}}{\langle y_1\rangle^{n-2s}}\frac{\lambda_2^{\frac{n-2s}{2}}}{\langle y_2\rangle^{n-2s}}dx\lesssim&\ R^{2s-n}.
    \label{V-U-4}
\end{align}
\end{lemma}

\begin{lemma}\label{Lemma B.1}
Suppose $n\geq 6s$ and $R \gg 1$, then
\begin{align}
    \int_{\R^n} U_1^{p-1}w^2_{1,1} &\approx \int_{|y_1|\leq R} \frac{\lambda_1^{2s}}{\langle y_1\rangle^{4s}}\frac{\lambda_1^{n-2s}R^{4s-2n}}{\langle y_1\rangle^{4s}}dx
    \lesssim R^{-n-4s}, \label{sigma(p-1)rho2-1}\\
    \int_{\R^n} U_1^{p-1}w^2_{2,1} &\approx \int_{|y_1|\geq R} \frac{\lambda_1^{2s}}{\langle y_1\rangle^{4s}}\frac{\lambda_1^{n-2s}R^{-8s}}{\langle y_1\rangle^{2n-8s}}dx\lesssim   R^{-n-4s},\label{sigma(p-1)rho2-2}\\
    \int_{\R^n} U_2^{p-1}w^2_{1,1}&\approx \int_{|y_1|\leq R} \frac{\lambda_2^{2s}}{\langle y_2\rangle^{4s}}\frac{\lambda_1^{n-2s}R^{4s-2n}}{\langle y_1\rangle^{4s}}dx\lesssim  R^{-n-4s},\label{sigma(p-1)rho2-3}\\
    \int_{\R^n} U_2^{p-1}w^2_{2,1} &\approx \int_{|y_1|\geq R} \frac{\lambda_2^{2s}}{\langle y_2\rangle^{4s}}\frac{\lambda_1^{n-2s}R^{-8s}}{\langle y_1\rangle^{2n-8s}}dx\lesssim R^{-n-4s}.
    \label{sigma(p-1)rho2-4}
\end{align}
\end{lemma}

\begin{lemma}\label{Lemma B.3}
Suppose $n=6s$ and $R\gg 1$,   then
\begin{align}
        \sum_{j=1}^2\int_{\R^n} U_1(x)U_2(x)\frac{\lambda_j^{2s}R^{-4s}}{\langle y_j\rangle^{2s}}dx\lesssim&\ R^{-8}\log R,
        \label{6-dim-1}
\end{align}
where $y_i = \lambda_i(x-z_i)$ is the $z_i-$centered co-ordinate for $i=1,2.$
\end{lemma}

\subsection*{Acknowledgements}The author would like to thank Prof. Alessio Figalli, Prof. Mateusz Kwaśniciki, Federico Glaudo, and the referees for their valuable comments and suggestions.

\subsection*{Data Availability Statement}
Data sharing is not applicable to this article as no datasets were generated or analyzed during the current study.
\bibliographystyle{plain}
\bibliography{refs}

\newpage

\end{document}